\documentclass[a4paper,12pt,intlimits,oneside]{amsart}

\usepackage{enumerate}
\usepackage{amsfonts,amsmath}

\usepackage{latexsym,amssymb}

\usepackage{mathrsfs}

\allowdisplaybreaks

\newcommand{\comment}[1]{}

\newcommand{\R}{{\mathbb R}}

\def\X{{\mathcal X}}

\def\a{\mathfrak a}

\def\m{\mathfrak m}

\def\M{{\mathcal M}}

\def\X{{\mathcal X}}

\def\sign{{\mbox{\rm  sign}}\,}

\newcommand{\ackname}{Acknowledgements}
\makeatletter
\if@titlepage
\newenvironment{acknowledgement}{%
	\titlepage
	\null\vfil
	\@beginparpenalty\@lowpenalty
	\begin{center}%
		\bfseries \ackname
		\@endparpenalty\@M
\end{center}}%
{\par\vfil\null\endtitlepage}
\else
\newenvironment{acknowledgement}{%
	\if@twocolumn
	\section*{\ackname}%
	\else
	\begin{center}%
		{\bfseries \ackname\vspace{0em}\vspace{\z@}}%
	\end{center}%
	\quotation
	\fi}
\fi
\makeatother

\begin{document}
	
	\title[]{Generalized Calder\'on-Zygmund operators on the Hardy space $H^1_\rho(\X)$}         % Enter your title between curly braces

	\author{Luong Dang Ky}
	%\dedicatory{Dedicated to Professor Aline Bonami on the occasion of her 80th birthday}
	\address{Department of Education, Quy Nhon University, 170 An Duong Vuong, Quy Nhon, Binh Dinh, Vietnam} 
	\email{{\tt luongdangky@qnu.edu.vn}}


%	\begin{abstract}
%		Let $(\X, d,\mu)$ be an RD-space and $\rho$ be an admissible function on $\X$. In this paper, we give necessary and sufficient conditions for the boundedness of a new class of generalized Calder\'on-Zygmund operators of log-Dini type associated to $\rho$ on the Hardy space $H^1_\rho(\X)$ introduced by Yang and Zhou \cite{YZ}. Our results extend some recent results in \cite{BDK, BLL, DLPV22, DLPV23, MSTZ}.		 
%	\end{abstract}

	\begin{abstract}
		Let $(\X, d,\mu)$ be an RD-space, and let $\rho$ be an admissible function on $\X$. We establish necessary and sufficient conditions for the boundedness of a new class of generalized Calder\'on-Zygmund operators of log-Dini type on the Hardy space $H^1_\rho(\X)$, introduced by Yang and Zhou. Our results extend and unify some recent results, providing further insights into the study of singular integral operators in this setting.
	\end{abstract}

	\maketitle
	\newtheorem{theorem}{Theorem}[section]
	\newtheorem{lemma}{Lemma}[section]
	\newtheorem{proposition}{Proposition}[section]
	\newtheorem{remark}{Remark}[section]
	\newtheorem{corollary}{Corollary}[section]
	\newtheorem{definition}{Definition}[section]
	\newtheorem{example}{Example}[section]
	\numberwithin{equation}{section}
	\newtheorem{Theorem}{Theorem}[section]
	\newtheorem{Lemma}{Lemma}[section]
	\newtheorem{Proposition}{Proposition}[section]
	\newtheorem{Remark}{Remark}[section]
	\newtheorem{Corollary}{Corollary}[section]
	\newtheorem{Definition}{Definition}[section]
	\newtheorem{Example}{Example}[section]
	\newtheorem*{theoremjj}{Theorem J-J}
	\newtheorem*{theorema}{Theorem A}
	\newtheorem*{theoremb}{Theorem B}
	\newtheorem*{theoremc}{Theorem C}
	\newtheorem*{conjecture}{Conjecture}
	\newtheorem*{open question}{Open question}

	\section{Introduction and statement of the results}	
	
	The theory of classical Hardy spaces $H^p(\R^n)$ plays an important role in many areas of analysis and PDEs. When studying the theory of singular integral operators and their applications, the Hardy spaces $H^p(\R^n)$ are good substitutes for the Lebesgue spaces $L^p(\R^n)$ when $p\in (0,1]$. For example, when $p\in (0,1]$, it is well known that the classical Riesz transforms $\{R_j\}_{j=1}^n$ are not bounded on $L^p(\R^n)$, however, they are bounded on $H^p(\R^n)$. Unfortunately, when $p\in (0,1]$, the principle that $H^p(\R^n)$ is like  $L^p(\R^n)$ breaks down at some key points namely: (1) $H^p(\R^n)$ does not contain the Schwartz class of rapidly decreasing test functions $\mathscr S(\R^n)$; (2) pseudo-differential operators are not bounded on $H^p(\R^n)$; (3) if $\phi\in \mathscr S(\R^n)$ and $f\in H^p(\R^n)$ then $\phi f$ is not necessarily in $H^p(\R^n)$. These, together with some applications in PDEs, motivated Goldberg \cite{Go} to introduce and study a localized version of the Hardy spaces $H^p(\R^n)$ that is denoted by $h^p(\R^n)$. Recently, a theory of Hardy spaces associated with differential operators $H^p_L(\R^n)$ was also introduced and developed by many authors (see, for example, \cite{BDK, DZ, DZ03, HLMMY, YZ}).
	
	When studying the boundedness of classical Calder\'on-Zygmund operators $T$ on the Hardy spaces $H^p(\R^n)$, it is well known that the necessary and sufficient condition for the boundedness of $T$ on $H^1(\R^n)$ is the condition $T^*1=0$, that is $\int_{\R^n} Tf(x) dx=0$ for all $f\in H^1(\R^n)$. In recent years, there has been an increasing interest in the study of boundedness of generalized Calder\'on-Zygmund operators $T$ on the Hardy type spaces (see, for example, \cite{BDK, BLL, BL, DLPV22, DLPV23, DHZ, MSTZ}). A natural question arises is what can one say about necessary and sufficient conditions for the boundedness of these generalized Calder\'on-Zygmund operators $T$ on the classical local Hardy spaces $h^p(\R^n)$ and on the Hardy spaces associated with differential operators $H^p_L(\R^n)$?
	
	Let $\varepsilon\in (0,\infty)$ and $\delta\in (0,1]$. Following Ding, Han and Zhu \cite{DHZ}, we say that $T$ is an inhomogeneous $(\varepsilon,\delta)$-Calder\'on-Zygmund operator if $T$ is bounded on $L^2(\R^n)$ and its kernel $K(x,y)$ satisfies the following conditions: there exists a positive constant $C$ such that for all $x\ne y$,
	$$|K(x,y)|\leq C\min\left\{\frac{1}{|x-y|^n}, \frac{1}{|x-y|^{n+\varepsilon}}\right\};$$
	and for all $x,x',y\in\X$ with $|x-x'|\leq \frac{|x-y|}{2}$,
	$$|K(x,y)-K(x',y)|+|K(y,x)-K(y,x')|\leq \frac{C}{|x-y|^{n}} \left(\frac{|x-x'|}{|x-y|}\right)^\delta.$$
	Recently, in \cite{BL, DLPV22, DLPV23, DHZ}, the authors provided necessary and sufficient conditions for the boundedness of the inhomogeneous $(\varepsilon,\delta)$-Calder\'on-Zygmund operators $T$ on the local Hardy spaces $h^p(\R^n)$ and their dual spaces.
	
	Let $\delta\in (0,1]$ and $L=-\Delta+V$ be a Schr\"odinger operator on $\R^n$ with $n\geq 3$ and the nonnegative potential $V$ belongs to the reverse H\"older class $RH_{n/2}(\R^n)$, one defines the critical radius function $\rho_L:\R^n\to [0,\infty)$ associated to $L$ as follows
	\begin{equation}\label{15:30, 28/10/2024}
		\rho_L(x)=\sup\left\{r>0:\frac{1}{r^{n-2}}\int_{B(x,r)} V(y)dy\leq 1\right\}.
	\end{equation}
	Following Ma, Stinga, Torrea and Zhang \cite{MSTZ}, we say that $T$ is a $\delta$-Schr\"odinger-Calder\'on-Zygmund operator if $T$ is bounded on $L^2(\R^n)$ and its kernel $K(x,y)$ satisfies the following conditions: for each $N>0$ there exists a positive constant $C(N)$ such that for all $x\ne y$,
	$$|K(x,y)|\leq \frac{C(N)}{|x-y|^n}\left(1+\frac{|x-y|}{\rho_L(x)}\right)^{-N};$$
	and there exists a positive constant $C$ such that for all $x,x',y\in\X$ with $|x-x'|\leq \frac{|x-y|}{2}$,
	$$|K(x,y)-K(x',y)|+|K(y,x)-K(y,x')|\leq  \frac{C}{|x-y|^{n}}\left(\frac{|x-x'|}{|x-y|}\right)^\delta.$$
	Recently, in \cite{BDK, BLL, MSTZ}, the authors provided necessary and sufficient conditions for the boundedness of the $\delta$-Schr\"odinger-Calder\'on-Zygmund operators $T$ on the Hardy spaces associated with Schr\"odinger operators $H^p_L(\R^n)$ and their dual spaces.
	
	Let $(\X, d,\mu)$ be an RD-space, that is, $(\X, d,\mu)$ is a space of homogeneous type in the sense of Coifman and Weiss \cite{CW} with the additional property that a reverse doubling property holds in $\X$ (see Section \ref{11:15, 31/10/2023}). Typical examples of RD-spaces include Euclidean spaces, Euclidean spaces with weighted measures satisfying the doubling property, and Carnot-Carath\'eodory spaces with doubling measures. We refer to the seminal paper of Han, M\"uller and Yang \cite{HMY} (see also \cite{GLY, LY, YYZ, YZ}) for a systematic study of the theory of function spaces in harmonic analysis on RD-spaces. Let $\rho$ be an admissible function on $\X$ (see Definition \ref{10:09, 30/11/2023}), in \cite{YZ}, Yang and Zhou introduced and developed a theory of localized Hardy spaces $H^1_\rho(\X)$ associated with $\rho$.  When $\X=\R^n$ with $n\geq 3$, $V$ belongs to the reverse H\"older class $RH_{n/2}(\R^n)$, and $\rho=\rho_L$ is the critical radius function associated to the Schr\"odinger operator $L=-\Delta+V$ (see \eqref{15:30, 28/10/2024}), it is well known (see \cite{YZ}) that the Hardy space $H^1_\rho(\X)$ coincides with the Hardy space $H^1_L(\R^n)$. This theory of the Hardy spaces $H^1_\rho(\X)$ has not only applications for Schr\"odinger operator settings, but also a wide range of applications in other settings (see the seminal paper \cite{YZ}). Recently, Bui, Duong and Ky \cite{BDK} extended the work of Yang and Zhou \cite{YZ} by developing a theory of localized Hardy spaces $H^p_\rho(\X)$ with $p\in (\frac{\mathfrak{n}}{\mathfrak{n}+1}, 1]$, where $\mathfrak{n}$ is the doubling order of $\mu$ (see \eqref{16:53, 14/10/2023}). 
	
	Motivated by \cite{BDK, YZ} and recent studies concerning the boundedness of the $\delta$-Schr\"odinger-Calder\'on-Zygmund operators (see \cite{BDK, BLL, MSTZ}) and of the inhomogeneous $(\varepsilon,\delta)$-Calder\'on-Zygmund operators (see \cite{BL, DLPV22, DLPV23, DHZ}) on the Hardy type spaces and their dual spaces, in this paper, we introduce and study the boundedness of the generalized Calder\'on-Zygmund operators of log-Dini type associated to $\rho$ (see Definitions \ref{10:32, 02/11/2023} and \ref{20:36, 21/12/2023}) on the Hardy space $H^1_\rho(\X)$ of Yang and Zhou \cite{YZ}, and its dual space $BMO_\rho(\X)$. These generalized Calder\'on-Zygmund operators are more general than the $\delta$-Schr\"odinger-Calder\'on-Zygmund operators in \cite{BPQ, BHQ, BDK, BLL, MSTZ} and the inhomogeneous $(\varepsilon,\delta)$-Calder\'on-Zygmund operators in \cite{BL, DLPV22, DLPV23, DHZ}. 
	
	Recall that a modulus of continuity $\omega: [0,\infty)\to [0,\infty)$ (that is, $\omega$ is increasing, subadditive and $\omega(0)=0$) is said to be satisfying the Dini condition if
	\begin{equation}\label{10:08, 31/10/2023}
		\|\omega\|_{\text{Dini}}:=\int_{0}^{1} \omega(t)\frac{dt}{t}<\infty,
	\end{equation}
	or is said to be satisfying the log-Dini condition if
	\begin{equation}\label{10:09, 31/10/2023}
		\|\omega\|_{\text{log-Dini}}:=\int_{0}^{1} \omega(t)\log\frac{1}{t}\frac{dt}{t}<\infty.
	\end{equation}
	Let $\omega_1, \omega_2: [0,\infty)\to[0,\infty)$ be two modulus of continuity satisfying the Dini condition, $\rho$ be an admissible function on the RD-Space $(\X,d,\mu)$, and $\kappa$ be the quasi-metric constant of $d$ (see \eqref{16:19, 05/10/2023}). Based on the ideas from \cite{CW, DHZ, MSTZ}, we define the generalized Calder\'on-Zygmund operators associated to $\omega_1, \omega_2$ and $\rho$ as follows.

	\begin{definition}\label{10:32, 02/11/2023}
		A bounded linear operator $T: L^2(\X)\to L^2(\X)$ is called an $(\omega_1,\omega_2)_\rho$-Calder\'on-Zygmund operator if 
		$$Tf(x)=\int_{\X} K(x,y)f(y)d\mu(y)$$
		for all $f\in L^2(\X)$ with bounded	support and all $x\notin\,$supp$\,f$, where the kernel $K(x,y)$ satisfies the following conditions:
		\begin{enumerate}[\rm (i)]
			\item there exists a positive constant $C$ such that for all $x\ne y$,
			\begin{equation}\label{Calderon-Zygmund 1}
				|K(x,y)|\leq \frac{C}{\mu(B(x,d(x,y)))} \omega_1\left(\left(1+ \frac{d(x,y)}{\rho(x)}\right)^{-1}\right);
			\end{equation}
		
			\item for all $x,x',y\in\X$ with $d(x,x')\leq \frac{d(x,y)}{2\kappa}$,
			\begin{equation}\label{Calderon-Zygmund 2}
				|K(x,y)-K(x',y)|+|K(y,x)-K(y,x')|\leq \frac{C}{\mu(B(x,d(x,y)))} \omega_2\left(\frac{d(x,x')}{d(x,y)}\right).
			\end{equation}			
		\end{enumerate}
	\end{definition}

	\begin{remark}\label{17:15, 24/10/2023}
		\begin{enumerate}[\rm (i)]
			\item The log-Dini condition \eqref{10:09, 31/10/2023} implies the Dini condition \eqref{10:08, 31/10/2023}. 
			
			\item\label{17:16, 24/10/2023} Let $k_0$ be as in Definition \ref{10:09, 30/11/2023} and $\omega_{1,k_0}(t):= \omega_1(t^{\frac{1}{k_0+1}})$ for all $t\in [0,\infty)$, then
			$$\|\omega_{1,k_0}\|_{\text{Dini}}= (k_0+1)\|\omega_{1}\|_{\text{Dini}}$$
			and
			$$\|\omega_{1,k_0}\|_{\text{log-Dini}}= (k_0+1)^2\|\omega_{1}\|_{\text{log-Dini}}.$$
			Moreover, by Remark \ref{11:36, 02/11/2023}, we see that if $T$ is an $(\omega_1,\omega_2)_\rho$-Calder\'on-Zygmund operator, then $T^*$ is an $(\omega_{1,k_0},\omega_2)_\rho$-Calder\'on-Zygmund operator.
			
			\item The condition \eqref{Calderon-Zygmund 2} implies the standard H\"ormander condition. More precisely, there exist two constants $0<C_1<C_2$ such that for all $y, z\in\X$,
			\begin{align*}
				&\int_{d(x,z)>2\kappa d(y,z)}\left(\left|K(x,y)-K(x,z)\right|+ \left|K(y,x)-K(z,x)\right|\right)d\mu(x)\\
				&\hskip6cm \leq C_1 \sum_{j=1}^{\infty}\omega_2((2\kappa)^{-j})\leq C_2 \|\omega_2\|_{\text{Dini}}<\infty.
			\end{align*}
			Therefore, $T$ and $T^*$ are of weak type $(1,1)$, and thus are bounded on $L^q(\X)$ for every $q\in (1,\infty)$.
			
			\item When $\X=\R^n$ and $\rho=1$, if $T$ is an inhomogeneous $(\varepsilon,\delta)$-Calder\'on-Zygmund operator as in \cite{BL, DLPV23, DHZ} for some $\varepsilon\in (0,\infty)$ and $\delta\in (0,1]$, then $T$ is a generalized $(\omega_1,\omega_2)_\rho$-Calder\'on-Zygmund of log-Dini type associated to $\omega_1(t)= (\alpha_1+\log\frac{1}{t})^{-\alpha_1}\chi_{(0,1]}(t)+\alpha_1^{-\alpha_1}\chi_{(1,\infty)}(t)$ and $\omega_2(t)= (\alpha_2+\log\frac{1}{t})^{-\alpha_2}\chi_{(0,1]}(t)+\alpha_2^{-\alpha_2}\chi_{(1,\infty)}(t)$ for any $\alpha_1,\alpha_2>2$. However, the reverse is not true in general. See also Remark \ref{21:02, 21/12/2023} for the case of the generalized inhomogeneous $(\varepsilon,\delta,s)$-Calder\'on-Zygmund operators studied in \cite{DLPV22}.
			
			\item When $\X=\R^n$ and $\rho=\rho_L$, if $T$ is a $\delta$-Schr\"odinger-Calder\'on-Zygmund operator as in \cite{MSTZ} for some $\delta\in (0,1]$, then $T$ is a generalized $(\omega_1,\omega_2)_\rho$-Calder\'on-Zygmund of log-Dini type associated to $\omega_1(t)= (\alpha_1+\log\frac{1}{t})^{-\alpha_1}\chi_{(0,1]}(t)+\alpha_1^{-\alpha_1}\chi_{(1,\infty)}(t)$ and $\omega_2(t)= (\alpha_2+\log\frac{1}{t})^{-\alpha_2}\chi_{(0,1]}(t)+\alpha_2^{-\alpha_2}\chi_{(1,\infty)}(t)$ for any $\alpha_1,\alpha_2>2$. However, the reverse is not true in general. See also Remark \ref{21:02, 21/12/2023} for the case of the generalized $(\delta,s)$-Schr\"odinger-Calder\'on-Zygmund operators studied in \cite{BPQ, BHQ, BDK, BLL}.
		\end{enumerate}	
	\end{remark}

	Now our first result can be stated as follows.
	
	\begin{theorem}\label{14:43, 02/11/2023}
			Let $T$ be an $(\omega_1,\omega_2)_\rho$-Calder\'on-Zygmund operator. Then:
		\begin{enumerate}[\rm (i)]
			\item $T$ is bounded from $H^1_\rho(\X)$ into $L^1(\X)$.
			
			\item $T$ is bounded from $H^1_\rho(\X)$ into $H^1(\X)$ if and only if $T^*1=0$ whenever $\omega_1$ and $\omega_2$ satisfy the log-Dini condition. 
		\end{enumerate}
	\end{theorem}

	From Theorem \ref{14:43, 02/11/2023} and Remark \ref{17:15, 24/10/2023}\eqref{17:16, 24/10/2023}, we obtain:

	\begin{corollary}\label{12:49, 04/12/2023}
		Under the same hypothesis as in Theorem \ref{14:43, 02/11/2023}. Then:
		\begin{enumerate}[\rm (i)]
			\item\label{12:50, 04/12/2023} \begin{enumerate}[$\bullet$]
				\item $T^*$ is bounded from $H^1_\rho(\X)$ into $L^1(\X)$;
				
				\item $T$ and $T^*$ are bounded from $L^\infty(\X)$ into $BMO_\rho(\X)$.
			\end{enumerate}
			
			\item \begin{enumerate}[$\bullet$]
				\item $T$ is bounded from $BMO(\X)$ into $BMO_\rho(\X)$ if and only if $T1=0$;
				
				\item $T^*$ is bounded from $BMO(\X)$ into $BMO_\rho(\X)$ if and only if $T^*1=0$
			\end{enumerate} 
			whenever $\omega_1$ and $\omega_2$ satisfy the log-Dini condition.
		\end{enumerate}
	\end{corollary}

	Following Bongioanni, Harboure and Salinas \cite{BHS}, we say that a function $f\in L^1_{\rm loc}(\X)$ is in $BMO^{\log}_\rho(\X)$ if
	$$\|f\|_{BMO^{\log}_\rho}=\sup_{B} \frac{\log(e+\frac{\rho(x_0)}{r})}{\mu(B)} \int_B \left|f(x)-\frac{1}{\mu(B)}\int_B f(y)d\mu(y)\right| d\mu(x)<\infty,$$
	where the supremum is taken over all balls $B=B(x_0,r)\subset\X$.

	Our main result is the following theorem.
	
	\begin{theorem}\label{18:02, 10/10/2023}
		Let $q\in (1,\infty]$ and $T$ be an $(\omega_1,\omega_2)_\rho$-Calder\'on-Zygmund operator with $\omega_1$ and $\omega_2$ satisfying the log-Dini condition. Then, the following three statements are equivalent:
		\begin{enumerate}[\rm (i)]
			\item\label{17:59, 10/10/2023} $T$ is bounded on $H^1_\rho(\X)$;
			
			\item\label{18:01, 10/10/2023} There exists a positive constant $C$ such that
			$$\left|\langle T^*1,\a\rangle\right|=\left|\int_{\X}T\a(x)d\mu(x)\right|\leq  \frac{C}{\log\left(e+\frac{\rho(x_0)}{r}\right)}$$
			for all $(H^1_\rho,q)$-atoms $\a$ related to the balls $B=B(x_0,r)$;
			
			\item\label{18:00, 10/10/2023} $T^*1\in BMO_\rho^{\log}(\X)$, that is, there exists a positive constant $C$ such that		
			$$\frac{1}{\mu(B)}\int_{B} \left|T^*1(x)-\frac{1}{\mu(B)}\int_B T^*1(y)d\mu(y)\right|d\mu(x)\leq  \frac{C}{\log\left(e+\frac{\rho(x_0)}{r}\right)}$$		
			for all balls $B=B(x_0,r)$.		
		\end{enumerate}
	\end{theorem}
	
	From Theorem \ref{18:02, 10/10/2023} and Remark \ref{17:15, 24/10/2023}\eqref{17:16, 24/10/2023}, we obtain:
	
	\begin{corollary}
		Under the same hypothesis as in Theorem \ref{18:02, 10/10/2023}. Then, the following three statements are equivalent:
		\begin{enumerate}[\rm (i)]
			\item $T$ is bounded on $BMO_\rho(\X)$;
			
			\item There exists a positive constant $C$ such that
			$$\left|\langle T1,\a\rangle\right|=\left|\int_{\X} T^*\a(x)d\mu(x)\right|\leq  \frac{C}{\log\left(e+\frac{\rho(x_0)}{r}\right)}$$
			for all $(H^1_\rho,q)$-atoms $\a$ related to the balls $B=B(x_0,r)$;
			
			\item $T1\in BMO_\rho^{\log}(\X)$, that is, there exists a positive constant $C$ such that		
			$$\frac{1}{\mu(B)}\int_{B} \left|T1(x)-\frac{1}{\mu(B)}\int_B T1(y)d\mu(y)\right|d\mu(x)\leq  \frac{C}{\log\left(e+\frac{\rho(x_0)}{r}\right)}$$		
			for all balls $B=B(x_0,r)$.		
		\end{enumerate}
	\end{corollary}
	
	\begin{remark}
		In order to study the boundedness of the $(\omega_1,\omega_2)_\rho$-Calder\'on-Zygmund operators on $H^1_\rho(\X)$, we need to introduce a new notion of molecules for $H^1_\rho(\X)$ (see Definition \ref{the definition for atoms and molecules of log-type}\eqref{09:28, 27/11/2023}). As an application, we obtain a new molecular characterization of $H^1_\rho(\X)$ (see Theorem \ref{the molecular characterization}). In Appendix, we shall also consider a wider class of the generalized $(\omega_1,\omega_2,s)_\rho$-Calder\'on-Zygmund operators with kernels satisfying another type	of regularity (see Definition \ref{20:36, 21/12/2023}). Our results generalize some recent results	in \cite{BDK, BLL, DLPV22, DLPV23, MSTZ}.
	\end{remark}

	The organization of the paper is as follows. In Section 2, we present some preliminaries on RD-spaces, admissible functions, Hardy type spaces, atoms and $BMO$ type spaces. In Section 3, we give a new molecular characterization of $H^1_\rho(\X)$ and some auxiliary lemmas. Section 4 is devoted to the proofs of the main theorems. Finally, in Section 5, we provide an appendix in which we state the main theorems for the generalized $(\omega_1,\omega_2,s)_\rho$-Calder\'on-Zygmund operators.
	
	 Throughout the whole paper, we always use $C$ to denote a positive constant that is independent of the main parameters involved, whose value may differ from line to line. If
	$f\leq C g$, we write $f\lesssim g$ and, if $f\lesssim g\lesssim f$, we then write $f\sim g$.  For any ball $B=B(x,r)$ and $t>0$, we use $tB:=B(x,tr)$. For any set $E\subset \X$, we denote by $\chi_E$ its characteristic function. For any index $q\in [1,\infty]$, we also denote by $q'$ its conjugate index, that is, $\frac{1}{q}+\frac{1}{q'}=1$. Let $\mathbb{Z}^+:=\{1, 2,\ldots\}$ and $\mathbb{Z}^+_0:=\mathbb{Z}^+\cup\{0\}$.
	
	\section{Preliminaries}\label{11:15, 31/10/2023}
	
	Let $d$ be a quasi-metric on a set $\X$, that is, $d$ is a nonnegative function on $\mathcal X\times \mathcal X$ satisfying
	\begin{enumerate}[(a)]
		\item $d(x,y)=d(y,x)$ for all $x, y\in \X$,
		\item $d(x,y)>0$ if and only if $x\ne y$,
		\item there exists a constant $\kappa\geq 1$ such that for all $x,y,z\in \mathcal X$,
		\begin{equation}\label{16:19, 05/10/2023}
			d(x,z)\leq \kappa(d(x,y)+ d(y,z)).
		\end{equation}
	\end{enumerate}
	
	A trip $(\mathcal X, d,\mu)$ is called a space of homogeneous type in the sense of Coifman and Weiss \cite{CW} if $\mu$ is a regular Borel measure satisfying doubling property, i.e., there exists a constant $C>1$ such that for all $x\in \mathcal X$ and $r>0$,
	$$\mu(B(x,2r))\leq C \mu(B(x,r)).$$
%	By the  geometric  structure  of spaces of homogeneous type (see \cite[Theorem 2]{MS}), we always assume that there exist two constants $\alpha_0\in (0,1]$ and $C>0$ such that for all $x, y, z\in \mathcal X$,
%	\begin{equation}\label{Macias-Segovia}
%		|d(x,z)-d(y,z)|\leq C (d(x,z)+d(y,z))^{1-\alpha_0} d(x,y)^{\alpha_0}.
%	\end{equation}
	Following Han, M\"uller and Yang \cite{HMY}, a space of homogeneous type $(\mathcal X, d,\mu)$ is called an RD-space if $\mu$ also satisfies reverse doubling property, i.e., there exists a constant $C>1$ such that for all $x\in \mathcal X$ and $r>0$,
	$$\mu(B(x,2r))\geq C \mu(B(x,r)).$$	
	Remark that the trip $(\mathcal X, d,\mu)$ is an RD-space if and only if  there exist constants $0<\mathfrak d \leq \mathfrak n$ and $C_0> 1$ such that for all $x\in\mathcal X$, $r>0$ and $\lambda\geq 1$,
	\begin{equation}\label{16:53, 14/10/2023}
		C_0^{-1} \lambda^{\mathfrak d} \mu(B(x,r)) \leq \mu(B(x,\lambda r))\leq C_0 \lambda^{\mathfrak n} \mu(B(x,r)).
	\end{equation}

	\begin{definition}\label{definition for test functions}
		Let $x_0\in\mathcal X$, $r>0$, $\beta\in (0,1]$ and $\gamma >0$. A function $\phi$ is said to belong to the space of test functions $\mathcal G(x_0,r,\beta,\gamma)$ if there exists a positive constant $C_\phi$ such that
		\begin{enumerate}[\rm (i)]
			\item $|\phi(x)| \leq  \frac{C_\phi}{\mu(B(x_0,r)) + \mu(B(x_0, d(x,x_0)))}\left(\frac{r}{r+ d(x,x_0)}\right)^\gamma$ for all $x\in\mathcal X$;
			
			\item $|\phi(x) - \phi(y)|\leq   \frac{C_\phi}{\mu(B(x_0,r)) + \mu(B(x_0, d(x,x_0)))}   \left(\frac{r}{r+ d(x,x_0)}\right)^\gamma \left(\frac{d(x,y)}{r+ d(x,x_0)}\right)^\beta$ for all $x,y\in \mathcal X$ satisfying that $d(x,y)\leq \frac{r + d(x,x_0)}{2\kappa}$.
			
		\end{enumerate}
		Moreover, for any $\phi\in \mathcal G(x_0,r,\beta,\gamma)$, we define its norm by
		$$\|\phi\|_{\mathcal G(x_0,r,\beta,\gamma)}:= \inf \left\{C_\phi: \mbox{{\rm (i)} and {\rm (ii)} hold}\right\}.$$
	\end{definition}

	In \cite{YZ}, Yang and Zhou introduced the notion of admissible functions as follows.

	\begin{definition}\label{10:09, 30/11/2023}
		A positive function $\rho$ on $\X$ is said to be admissible	if there exist positive constants $k_0$ and $c_0$ such that for all $x,y\in \X$,
		\begin{equation}\label{admissible function}
			\frac{1}{\rho(x)}\leq c_0 \frac{1}{\rho(y)}\left(1+\frac{d(x,y)}{\rho(y)}\right)^{k_0}.
		\end{equation}
	\end{definition}
	
	\begin{remark}\label{11:36, 02/11/2023}
		If $\rho$ is an admissible function on $\X$, then for all $x,y\in \X$,
		$$\left(1+\frac{d(x,y)}{\rho(y)}\right)^{-(k_0+1)}\lesssim \left(1+\frac{d(x,y)}{\rho(x)}\right)^{-1}\lesssim \left(1+\frac{d(x,y)}{\rho(y)}\right)^{-\frac{1}{k_0+1}}.$$
	\end{remark}
	
	Throughout the whole paper, we always assume that $\mathcal X$ is an RD-space with $\mu(\mathcal X)=\infty$, and $\rho$ is an admissible function on $\X$. Also we fix $x_0\in \X$.
	
	In Definition \ref{definition for test functions}, it is easy to see that $\mathcal G(x_0,1,\beta,\gamma)$ is a Banach space. Furthermore, for any  $x\in \mathcal X$ and $r>0$, we have $\mathcal G(x,r,\beta,\gamma)= \mathcal G(x_0,1,\beta,\gamma)$  with equivalent norms (but of course the constants are depending on $x$ and $r$). For simplicity, we write $\mathcal G(\beta,\gamma)$ instead of $\mathcal G(x_0,1,\beta,\gamma)$.
	
	Let $\epsilon\in (0,1]$ and $\beta,\gamma\in (0,\epsilon]$, we define the space $\mathcal G^\epsilon_0(\beta,\gamma)$ to be the completion of $\mathcal G(\epsilon,\epsilon)$ in $\mathcal G(\beta,\gamma)$, and denote by $(\mathcal G^\epsilon_0(\beta,\gamma))'$ the space of all continuous linear functionals on $\mathcal G^\epsilon_0(\beta,\gamma)$. We say that $f$ is a distribution if $f\in (\mathcal G^\epsilon_0(\beta,\gamma))'$. For a distribution $f$, the  grand maximal functions $\M(f)$ and  $\M_\rho(f)$ are defined by
	$$\M(f)(x) := \sup\{|\langle f,\phi \rangle|: \phi\in \mathcal G^\epsilon_0(\beta,\gamma), \|\phi\|_{\mathcal G(x,r,\beta,\gamma)}\leq 1\; \mbox{for some}\; r>0\}$$
	and
	$$\M_\rho(f)(x) := \sup\{|\langle f,\phi \rangle|: \phi\in \mathcal G^\epsilon_0(\beta,\gamma), \|\phi\|_{\mathcal G(x,r,\beta,\gamma)}\leq 1\; \mbox{for some}\; r\in (0,\rho(x))\}.$$
	
	\begin{definition}
		Let $\epsilon\in (0,1)$ and $\beta,\gamma\in (0,\epsilon)$.
		\begin{enumerate}[\rm (i)]
			\item The Hardy space $H^1(\mathcal X)$ is defined by
			$$H^1(\mathcal X) = \{f\in (\mathcal G^\epsilon_0(\beta,\gamma))': \|f\|_{H^1}:= \|\M( f)\|_{L^1}<\infty \}.$$
			\item The Hardy space $H^1_\rho(\mathcal X)$ is defined by
			$$H^1_\rho(\mathcal X) = \{f\in (\mathcal G^\epsilon_0(\beta,\gamma))': \|f\|_{H^1_\rho}:= \|\M_\rho( f)\|_{L^1}<\infty \}.$$
		\end{enumerate}
	\end{definition}

	\begin{definition}\label{14:38, 06/10/2023}
		Let $q\in (1,\infty]$.
		\begin{enumerate}[\rm (i)]
			\item A measurable function $ \mathfrak a$ is called an $(H^1,q)$-atom related to the ball $B$ if 
			\begin{enumerate}[\rm (a)]
				\item supp\,$\a\subset B$,
				\item $\|\a\|_{L^q}\leq \mu(B)^{\frac{1}{q}-1}$,
				\item $\int_{\mathcal X} \a(x) d\mu(x)=0$.
			\end{enumerate}
			
			\item\label{14:39, 06/10/2023} A measurable function $\mathfrak a$ is called an $(H^1_\rho,q)$-atom related to the ball $B=B(x_0,r)$ if $r <\rho(x_0)$ and $\mathfrak a$ satisfies {\rm (a)} and {\rm (b)}, and when $r < \frac{\rho(x_0)}{4}$, $\mathfrak a$ also satisfies {\rm (c)}.
			
		\end{enumerate}
		
	\end{definition}

	The following results are well known (see \cite[Remark 5.5]{GLY} and \cite[Theorem 3.2]{YZ}).
	
	\begin{theorem}\label{atomic decomposition}
		Let  $\epsilon\in (0,1)$, $\beta,\gamma\in (0,\epsilon)$ and $q\in(1,\infty]$. Then, we have:
		\begin{enumerate}[\rm (i)]
			\item The space $H^1(\mathcal X)$ coincides with the Hardy space $H^{1,q}_{\rm at}(\mathcal X)$ of Coifman and Weiss. More precisely, $f\in H^1(\X)$ if and only if $f$ can be written as  $f= \sum_{j=1}^\infty \lambda_j a_j$, where $\{a_j\}_{j=1}^\infty$ are $(H^1,q)$-atoms and $\{\lambda_j\}_{j=1}^\infty\in\ell^1$. Moreover, 
			$$\|f\|_{H^1}\sim \inf \left\{\sum_{j=1}^\infty |\lambda_j| : f= \sum_{j=1}^\infty \lambda_j a_j\right\}.$$
						
			\item\label{07:03, 14/10/2023} $f\in H^1_\rho(\X)$ if and only if $f$ can be written as  $f= \sum_{j=1}^\infty \lambda_j a_j$, where $\{a_j\}_{j=1}^\infty$ are $(H^1_\rho,q)$-atoms and $\{\lambda_j\}_{j=1}^\infty\in\ell^1$. Moreover, 
			$$\|f\|_{H^1_\rho}\sim \inf \left\{\sum_{j=1}^\infty |\lambda_j| : f= \sum_{j=1}^\infty \lambda_j a_j\right\}.$$
		\end{enumerate}	
	\end{theorem}

Here and in what follows, for any ball $B\subset \X$ and $f\in L^1_{\rm loc}(\X)$, we denote 
\begin{equation}
	f_B:= \frac{1}{\mu(B)}\int_B f(y)d\mu(y)\quad\text{and}\quad MO(f,B):= \frac{1}{\mu(B)}\int_B |f(x)- f_B|d\mu(x).
\end{equation}

\begin{definition}
	\begin{enumerate}[\rm (i)]
		\item  A function $f\in L^1_{\rm loc}(\X)$ is said to be in $BMO(\X)$ if
		$$\|f\|_{BMO}=\sup_{B} MO(f,B)<\infty,$$
		where the supremum is taken over all balls $B\subset\X$.
		
		\item  A function $f\in BMO(\X)$ is said to be in $BMO_\rho(\X)$ if
		$$\|f\|_{BMO_\rho}= \|f\|_{BMO} + \sup_B |f|_B<\infty,$$
		where the supremum is taken over all balls $B=B(x_0,r)\subset\X$ with $r\geq \rho(x_0)$.
		
		\item A function $f\in L^1_{\rm loc}(\X)$ is said to be in $BMO^{\log}_\rho(\X)$ if
		$$\|f\|_{BMO^{\log}_\rho}=\sup_{B} \log\left(e+\frac{\rho(x_0)}{r}\right) MO(f,B)<\infty,$$
		where the supremum is taken over all balls $B=B(x_0,r)\subset\X$.
	\end{enumerate}
\end{definition}

\begin{remark}\label{14:22, 18/10/2023}
	\begin{enumerate}[\rm (i)]
		\item\label{09:16, 29/10/2023} Let $p\in [1,\infty)$. Then, by \cite[Theorem 3.1]{Na},
		$$\left(\frac{1}{\mu(B)}\int_B |f(x)- f_B|^p d\mu(x)\right)^{1/p}\lesssim \|f\|_{BMO}$$
		for all $f\in BMO(\X)$ and all balls $B\subset \X$.
		
		\item\label{09:12, 29/10/2023} A standard argument $(\text{see also \cite[Lemma 2.1]{LY}})$ gives
		$$|f_B|\lesssim \log\left(e+\frac{\rho(x_0)}{r}\right)\|f\|_{BMO_\rho}$$
		for all $f\in BMO_\rho(\X)$ and all balls $B=B(x_0,r)\subset \X$.
		
		\item For any $f\in BMO^{\log}_\rho(\X)$, we have $\|f\|_{BMO}\leq \|f\|_{BMO^{\log}_\rho}$.
	\end{enumerate}
	
\end{remark}

The following results are well known (see \cite[Theorem B]{CW} and \cite[Theorem 2.1]{YYZ}).

\begin{theorem}\label{06:44, 14/10/2023}
	\begin{enumerate}[\rm (i)]
		\item The space $BMO(\X)$ is the dual space of $H^1(\X)$.
		
		\item\label{14:53, 13/12/2023} The space $BMO_\rho(\X)$ is the dual space of $H^1_\rho(\X)$.
	\end{enumerate}
\end{theorem}	

\section{Some auxiliary results}

Here and in what follows, for any ball $B$, we denote
\begin{equation}
	\text{$S_0(B):= 2\kappa B$ and $S_j(B):=(2\kappa)^{j+1}B\setminus (2\kappa)^j B$ for all $j\in\mathbb{Z}^+$.}
\end{equation}

\begin{definition}\label{the definition for atoms and molecules of log-type}
	Let $q\in (1,\infty]$ and let $\omega_1, \omega_2$ satisfy the log-Dini condition. 
	\begin{enumerate}[\rm (i)]
		\item A measurable function $\mathfrak a$ is called an $(H^1_\rho,q)$-atom of log-type related to the ball $B=B(x_0,r)$ if 
		\begin{enumerate}[\rm (a)]
			\item supp\,$\a\subset B$,
			
			\item $\|\a\|_{L^q}\leq \mu(B)^{\frac{1}{q}-1}$,
			
			\item  $\left|\int_{\X}\a(x)d\mu(x)\right|\leq  \frac{1}{\log\left(e+ \frac{\rho(x_0)}{r}\right)}$.
		\end{enumerate} 
	
		\item\label{09:28, 27/11/2023} A measurable function $\mathfrak m$ is called an $(H^1_\rho,q,\omega_1,\omega_2)$-molecule of log-type related to the ball $ B=B(x_0,r)$ if 
		\begin{enumerate}[\rm (a)]
			\item\label{17:10, 14/12/2023} $\|\m\|_{L^q(S_j(B))}\leq [\omega_1(((2\kappa)^{\frac{1}{k_0+1}})^{-j})+ \omega_2(\left(2\kappa\right)^{-j})]\mu((2\kappa)^j B)^{\frac{1}{q}-1}$ for all $j\in\mathbb{Z}^+_0$,
			
			\item  $\left|\int_{\X}\m(x)d\mu(x)\right|\leq  \frac{1}{\log\left(e+ \frac{\rho(x_0)}{r}\right)}$.
		\end{enumerate}
	\end{enumerate}
\end{definition}

\begin{remark}\label{12:09, 22/10/2023}
	\begin{enumerate}[\rm (i)]
		\item\label{14:29, 12/10/2023} If $\a$ is an $(H^1_\rho,q)$-atom related to the ball $B=B(x_0,r)$, then $\frac{1}{\log(e+4)}\a$ is an $(H^1_\rho,q)$-atom of log-type related to the ball $B$. Moreover, 
		$$\left|g_B \int_{\X} \a(x)d\mu(x)\right|\lesssim \frac{\|g\|_{BMO_\rho}}{\log\left(e+\frac{\rho(x_0)}{r}\right)}$$
		for all $g\in BMO_\rho(\X)$, where the positive constant $C$ is independent of $\a$ and $g$. 
		
		\item\label{20:22, 20/10/2023} If $\a$ is an $(H^1_\rho,q)$-atom of log-type related to the ball $B$, then $\min\{\omega_1(1)+\omega_2(1), 1\}\a$ is an $(H^1_\rho,q,\omega_1,\omega_2)$-molecule of log-type related to the ball $B$.
		
		\item\label{16:28, 21/10/2023} Since $\omega_i(t_0 t)\lesssim 2^{\lfloor \log_2 t_0\rfloor +1} \omega_i(t)$ for all $t_0>1$ and $t\geq 0$, we obtain that
		$$\|\omega_1\|_{\text{Dini}}\sim \sum_{j=0}^{\infty} \omega_1\left(\left((2\kappa)^{\frac{1}{k_0+1}}\right)^{-j}\right),\; \|\omega_2\|_{\text{Dini}}\sim \sum_{j=0}^{\infty} \omega_2\left(\left(2\kappa\right)^{-j}\right)$$
		and
		$$\|\omega_1\|_{\text{log-Dini}}\sim \sum_{j=0}^{\infty} (j+1) \omega_1\left(\left((2\kappa)^{\frac{1}{k_0+1}}\right)^{-j}\right),\; \|\omega_2\|_{\text{log-Dini}}\sim \sum_{j=0}^{\infty} (j+1) \omega_2\left(\left(2\kappa\right)^{-j}\right).$$
		
		\item Given two constants $c_1, c_2>1$, the condition (a) in Definition \ref{the definition for atoms and molecules of log-type}\eqref{09:28, 27/11/2023} can be replaced by 
		$$\|\m\|_{L^q(S_j(B))}\leq [\omega_1(c_1^{-j})+ \omega_2(c_2^{-j})]\mu((2\kappa)^j B)^{\frac{1}{q}-1}\;\text{for all $j\in\mathbb{Z}^+_0$}.$$
	\end{enumerate}
\end{remark}

Let $\varepsilon>0$. A modulus of continuity $\omega: [0,\infty)\to [0,\infty)$ is said to satisfy the $\varepsilon$-Dini condition if
\begin{equation}\label{09:07, 27/11/2023}
	\|\omega\|_{\text{$\varepsilon$-Dini}}:=\int_{0}^{1} \omega(t)\frac{dt}{t^{1+\varepsilon}}<\infty,
\end{equation}

\begin{remark}
	\begin{enumerate}[\rm (i)]
		\item If $\omega_1$ and $\omega_2$ satisfy the $\varepsilon$-Dini condition for some $\varepsilon>0$, then our notion of molecules coincides with the one in \cite{BDK, BLL, DLPV22, DLPV23}.
		
		\item The log-Dini condition \eqref{10:09, 31/10/2023} is weaker than the $\varepsilon$-Dini condition \eqref{09:07, 27/11/2023}. Thus, the notion of molecules in \cite{BDK, BLL, DLPV22, DLPV23} is not applicable in our setting.
	\end{enumerate}
\end{remark}

\begin{lemma}\label{14:51, 06/10/2023}
	Let $p\in [1,\infty)$. Then, there exists a positive constant $C$ such that
	$$\left(\frac{1}{\mu((2\kappa)^j B)}\int_{(2\kappa)^j B} |f(x)- f_{B}|^p d\mu(x)\right)^{1/p}\leq C (j+1) \|f\|_{BMO}$$
	for all $f\in BMO(\X)$, all balls $B$, and all $j\in \mathbb Z^+_0$. 
\end{lemma}

\begin{proof}
	The case $j=0$ is trivial by Remark \ref{14:22, 18/10/2023}\eqref{09:16, 29/10/2023}. When $j\in\mathbb{Z}^+$, also by Remark \ref{14:22, 18/10/2023}\eqref{09:16, 29/10/2023}, we get
	\begin{align*}
		&\left(\frac{1}{\mu((2\kappa)^j B)}\int_{(2\kappa)^j B} |f(x)- f_{B}|^p d\mu(x)\right)^{1/p}=\frac{\|f-f_B\|_{L^p((2\kappa)^j B)}}{\mu((2\kappa)^j B)^{1/p}}\\
		&\hskip3.2cm \leq \frac{\|f-f_{(2\kappa)^j B}\|_{L^p((2\kappa)^j B)} }{\mu((2\kappa)^j B)^{1/p}}+\sum_{i=1}^{j} \frac{\|f_{(2\kappa)^i B}-f_{(2\kappa)^{i-1} B}\|_{L^p((2\kappa)^j B)}}{\mu((2\kappa)^j B)^{1/p}}\\
		&\hskip3.2cm \lesssim \|f\|_{BMO}+ \sum_{i=1}^{j} \frac{\mu((2\kappa)^i B)}{\mu((2\kappa)^{i-1} B)} MO(f, (2\kappa)^i B)\lesssim (j+1)\|f\|_{BMO}.
	\end{align*}
\end{proof}

\begin{lemma}\label{07:01, 14/10/2023}
	Let $q\in (1,\infty]$ and let $\omega_1, \omega_2$ satisfy the log-Dini condition. Then, there exists a positive constant $C$ such that
	$$\|\m\|_{H^1_\rho}\leq C$$
	for all $(H^1_\rho,q,\omega_1,\omega_2)$-molecules of log-type $\m$.
\end{lemma}

\begin{proof}
	By \cite[Theorem 1.1]{DHK}, it suffices to prove that 
	\begin{equation}\label{06:54, 14/10/2023}
		\left|\int_{\X} \m(x) g(x) d\mu(x) \right|\lesssim \|g\|_{BMO_\rho}
	\end{equation}
	for all $(H^1_\rho,q,\omega_1,\omega_2)$-molecules of log-type $\m$ related to the balls $B=B(x_0, r)$ and all $g\in C_c(\X)$. Indeed, it follows from Remark \ref{14:22, 18/10/2023}\eqref{09:12, 29/10/2023}, the H\"older inequality, Lemma \ref{14:51, 06/10/2023} and Remark \ref{12:09, 22/10/2023}\eqref{16:28, 21/10/2023} that
	\begin{align*}
		&\left|\int_{\X} \m(x) g(x) d\mu(x) \right|\leq |g_B|\left|\int_{\X} \m(x) d\mu(x) \right|+ \left|\int_{\X} \m(x) (g(x)-g_B) d\mu(x) \right| \\			
		&\hskip0.5cm\lesssim \log\left(e+\frac{\rho(x_0)}{r}\right)\|g\|_{BMO_\rho} \left|\int_{\X} \m(x) d\mu(x) \right| +\sum_{j=0}^\infty \|\m\|_{L^q(S_j(B))} \|g-g_B\|_{L^{q'}(S_j(B))}\\
		&\hskip0.5cm\lesssim \|g\|_{BMO_\rho} + \sum_{j=0}^\infty [\omega_1(((2\kappa)^{\frac{1}{k_0+1}})^{-j})+ \omega_2(\left(2\kappa\right)^{-j})] \mu((2\kappa)^{j}B)^{\frac{1}{q}-1}\times\\
		& \hskip9cm \times\mu((2\kappa)^{j+1}B)^{\frac{1}{q'}} (j+2)\|g\|_{BMO}\\
		&\hskip0.5cm\lesssim \|g\|_{BMO_\rho} +   \|g\|_{BMO} (\|\omega_1\|_{\text{log-Dini}}+ \|\omega_2\|_{\text{log-Dini}}) \lesssim \|g\|_{BMO_\rho}.
	\end{align*}
	This proves \eqref{06:54, 14/10/2023}, and thus completes the proof of Lemma \ref{07:01, 14/10/2023}.
	
\end{proof}

As a consequence of Lemma \ref{07:01, 14/10/2023}, Remark \ref{12:09, 22/10/2023}(\eqref{14:29, 12/10/2023}-\eqref{20:22, 20/10/2023}) and Theorem \ref{atomic decomposition}\eqref{07:03, 14/10/2023}, we obtain the following new molecular characterization of $H^1_\rho(\X)$.

\begin{theorem}\label{the molecular characterization}
	Let $q\in (1,\infty]$ and let $\omega_1,\omega_2$ satisfy the log-Dini condition. Then, $f\in H^1_\rho(\X)$ if and only if $f$ can be written as  $f= \sum_{j=1}^\infty \lambda_j \m_j$, where $\{\m_j\}_{j=1}^\infty$ are $(H^1_\rho,q,\omega_1,\omega_2)$-molecules of log-type and $\{\lambda_j\}_{j=1}^\infty\in\ell^1$. Moreover,
	$$\|f\|_{H^1_\rho}\sim \inf \left\{\sum_{j=1}^\infty |\lambda_j| : f= \sum_{j=1}^\infty \lambda_j \m_j\right\}.$$
\end{theorem}

	Let $p\in [1,\infty)$. A function $f\in L^1_{\rm loc}(\X)$ is said to be in $BMO^{\log}_{\rho,p}(\X)$ if
	$$\|f\|_{BMO^{\log}_{\rho,p}}=\sup_{B} \log\left(e+\frac{\rho(x_0)}{r}\right)\left(\frac{1}{\mu(B)} \int_B \left|f(x)-f_B\right|^p d\mu(x)\right)^{1/p}<\infty,$$
	where the supremum is taken over all balls $B=B(x_0,r)\subset\X$.

\begin{lemma}\label{08:21, 15/10/2023}
	Let $p\in [1,\infty)$. Then, $BMO^{\log}_{\rho,p}(\X)= BMO^{\log}_{\rho}(\X)$ with equivalent norms.
\end{lemma}

\begin{proof}
	Let $\phi_\rho(x,r)=\frac{1}{\log\left(e+\frac{\rho(x)}{r}\right)}$ for all $(x,r)\in \X\times (0,\infty)$. Then, by \cite[Theorem 3.1]{Na}, it suffices to prove that
	\begin{equation}\label{11:26, 14/10/2023}
		\frac{\phi_\rho(x,r)}{\phi_\rho(x,r_0)}\sim 1
	\end{equation}
	for all $\frac{1}{2}\leq \frac{r}{r_0}\leq 2$, and
	\begin{equation}\label{11:27, 14/10/2023}
		\phi_\rho(x,r)\lesssim \phi_\rho(x_0,r_0)
	\end{equation}
	for all $B(x,r)\subset B(x_0,r_0)$. Indeed, from $1\leq\frac{\log(e+2t)}{\log(e+t)}\leq 2$ for all $t> 0$, we get
	$$\frac{1}{2}\leq \frac{\phi_\rho(x,r)}{\phi_\rho(x,r_0)}=\frac{\log\left(e+\frac{\rho(x)}{r_0}\right)}{\log\left(e+\frac{\rho(x)}{r}\right)}\leq 2$$
	for all $\frac{1}{2}\leq \frac{r}{r_0}\leq 2$. This proves that \eqref{11:26, 14/10/2023} holds.
	
	In order to establish \eqref{11:27, 14/10/2023}, we show that for all $B(x,r)\subset B(x_0,r_0)$, 
	$$\phi_\rho(x,r)\leq c_0 2^{k_0} 2\kappa C_0^{1/{\mathfrak d}}  \phi_\rho(x_0,r_0),$$
	or equivalently,	
	\begin{equation}\label{17:23, 14/10/2023}
		\log\left(e+\frac{\rho(x_0)}{r_0}\right)\leq c_0 2^{k_0} 2\kappa C_0^{1/{\mathfrak d}} \log\left(e+\frac{\rho(x)}{r}\right),
	\end{equation}
	where $c_0, k_0$ are as in \eqref{admissible function}, and $C_0,\mathfrak d$ are as in \eqref{16:53, 14/10/2023}. Indeed, \eqref{17:23, 14/10/2023} is obvious if $r_0\geq\rho(x_0)$. When $0< r_0<\rho(x_0)$, it follows from \eqref{admissible function} and $B(x,r)\subset B(x_0,r_0)$ that
	\begin{equation}\label{17:14, 14/10/2023}
		\rho(x_0)\leq c_0 \left(1+\frac{d(x,x_0)}{\rho(x_0)}\right)^{k_0}\rho(x)\leq c_0 2^{k_0} \rho(x)
	\end{equation}
	and
	\begin{equation}\label{17:15, 14/10/2023}
		B(x,r)\subset B(x, 2\kappa r_0).
	\end{equation}
	Let us consider the following two cases:
	
	{\sl Case 1:} $r\leq 2\kappa r_0$. Then, by \eqref{17:14, 14/10/2023},
	$$\log\left(e+\frac{\rho(x_0)}{r_0}\right)\leq \log\left(e+c_0 2^{k_0} 2\kappa\frac{\rho(x)}{r}\right) \leq c_0 2^{k_0} 2\kappa C_0^{1/{\mathfrak d}} \log\left(e+\frac{\rho(x)}{r}\right).$$

	{\sl Case 2:} $r= 2\kappa r_0 \lambda$ with $\lambda>1$. Then, by \eqref{17:15, 14/10/2023} and \eqref{16:53, 14/10/2023},
	$$\mu(B(x,2\kappa r_0))= \mu(B(x,r)) =\mu(B(x, 2\kappa r_0 \lambda))\geq C_0^{-1} \lambda^{\mathfrak d}\mu(B(x,2\kappa r_0)).$$
	This implies that
	$$r= 2\kappa r_0 \lambda\leq 2\kappa r_0 C_0^{1/{\mathfrak d}}.$$
	Therefore, by \eqref{17:14, 14/10/2023}, 
	$$\log\left(e+\frac{\rho(x_0)}{r_0}\right)\leq \log\left(e+c_0 2^{k_0} 2\kappa C_0^{1/{\mathfrak d}}\frac{\rho(x)}{r}\right) \leq c_0 2^{k_0} 2\kappa C_0^{1/{\mathfrak d}} \log\left(e+\frac{\rho(x)}{r}\right),$$
	which proves \eqref{17:23, 14/10/2023}, and thus \eqref{11:27, 14/10/2023} holds. This completes the proof of Lemma \ref{08:21, 15/10/2023}.

\end{proof}

	\begin{lemma}\label{17:53, 11/10/2023}
		Let $s\in [1,\infty)$ and $q\in (s,\infty]$. Then, there exists a positive constant $C$ such that
		\begin{enumerate}[\rm (i)]
			\item\label{14:01, 12/10/2023} for any $(H^1_\rho,q)$-atom $\a$ related to the ball $B$ and any $f\in BMO(\X)$,
			$$\|\a(f-f_B)\|_{L^s}\leq C\mu(B)^{\frac{1}{s}-1}\|f\|_{BMO};$$
			
			\item\label{14:00, 12/10/2023} for any $(H^1_\rho,q)$-atom $\a$ related to the ball $B$ and any $f,g\in BMO(\X)$,
			$$\|\a(f-f_B)(g-g_B)\|_{L^1}\leq C \|f\|_{BMO}\|g\|_{BMO}.$$
		\end{enumerate}
	\end{lemma}
	
	\begin{proof}
		\eqref{14:01, 12/10/2023} Let $p\in [s,\infty)$ be such that $\frac{1}{s}=\frac{1}{q}+\frac{1}{p}$. Then, by the generalized H\"older inequality and Remark \ref{14:22, 18/10/2023}\eqref{09:16, 29/10/2023},
		\begin{align*}
			\|\a(f-f_B)\|_{L^s}&=\|\a(f-f_B)\|_{L^s(B)}\\
			&\leq \|\a\|_{L^q(B)}\|f-f_B\|_{L^p(B)}\\
			&\lesssim \mu(B)^{\frac{1}{q}-1} \mu(B)^{\frac{1}{p}}\|f\|_{BMO}=  \mu(B)^{\frac{1}{s}-1}\|f\|_{BMO}.
		\end{align*}

		\eqref{14:00, 12/10/2023} It follows from the generalized H\"older inequality and Remark \ref{14:22, 18/10/2023}\eqref{09:16, 29/10/2023} that
		\begin{align*}
			\|\a(f-f_B)(g-g_B)\|_{L^1}&=\|\a(f-f_B)(g-g_B)\|_{L^1(B)}\\
			&\leq \|\a\|_{L^q(B)}\|f-f_B\|_{L^{2q'}(B)}\|g-g_B\|_{L^{2q'}(B)}\\
			&\lesssim \mu(B)^{\frac{1}{q}-1}  \mu(B)^{\frac{1}{2q'}}\|f\|_{BMO} \mu(B)^{\frac{1}{2q'}}\|g\|_{BMO}\\
			&\lesssim \|f\|_{BMO}\|g\|_{BMO}.
		\end{align*}
		
	\end{proof}

For any $x_0\in\X$, we define $g_{x_0}:\X\to \R$ as follows
\begin{equation}\label{17:18, 11/10/2023}
	g_{x_0}(x)=\max\left\{0, 1+ \log\frac{\rho(x_0)}{d(x,x_0)}\right\}.
\end{equation}
Then, we have the following lemma whose proof will be given in Appendix (see Proposition \ref{14:47, 18/10/2023}). 

\begin{lemma}\label{17:37, 11/10/2023}
	There exists a positive constant $C$ such that for all $x_0\in\X$,
	$$\|g_{x_0}\|_{BMO_\rho}\leq C.$$
	Moreover, if $0<r<\rho(x_0)$ then $\log\left(e+\frac{\rho(x_0)}{r}\right)\leq 2 g_{x_0}(x)$ for all $x\in B(x_0,r)$.
\end{lemma}

It should be pointed out that a similar result to the above lemma was obtained in \cite[Lemma 2.2]{LY}, but there the authors required $d$ is a metric and $\mathfrak{d}=\mathfrak{n}$, where $\mathfrak{d}$ and $\mathfrak{n}$ are as in \eqref{16:53, 14/10/2023}.

	\begin{lemma}\label{07:01, 12/10/2023}
		Let $q\in (1,\infty]$. Then, a $BMO$-function $f$ is in $BMO_{\rho}^{\log}(\X)$ if and only if $\a(f-f_{B})\in H^1_\rho(\X)$ for all $(H^1_\rho,q)$-atoms $\a$ related to the balls $B$. Moreover,
		$$\|f\|_{BMO_{\rho}^{\log}}\sim \|f\|_{BMO}+ \sup_{\text{$(H^1_\rho,q)$-atom $\a$ related to $B$}} \|\a(f-f_{B})\|_{H^1_\rho}.$$		
	\end{lemma}
	
	\begin{proof}
		Suppose that $f\in BMO_{\rho}^{\log}(\X)$. Let $s:= \min\{2,\frac{q+1}{2}\}\in (1,q)$. Then, for any $(H^1_\rho,q)$-atom $\a$ related to the ball $B=B(x_0,r)$, Lemma \ref{17:53, 11/10/2023}\eqref{14:01, 12/10/2023} gives
		\begin{equation}\label{06:55, 12/10/2023}
			\|\a(f-f_{B})\|_{L^{s}}\lesssim \|f\|_{BMO}\mu(B)^{\frac{1}{s}-1}\lesssim \|f\|_{BMO_\rho^{\log}}\mu(B)^{\frac{1}{s}-1}.
		\end{equation}
		Moreover, from the H\"older inequality and Lemma \ref{08:21, 15/10/2023}, we get
		\begin{align*}
			\left|\int_{\X} \a(x)(f(x)-f_B)d\mu(x)\right|&\leq \|\a(f-f_B)\|_{L^1(B)}\\
			&\leq \|\a\|_{L^q(B)}\|f-f_B\|_{L^{q'}(B)}\\
			&\lesssim \mu(B)^{\frac{1}{q}-1} \mu(B)^{\frac{1}{q'}} \frac{1}{\log\left(e+\frac{\rho(x_0)}{r}\right)}\|f\|_{BMO_\rho^{\log}}\\
			&\lesssim \|f\|_{BMO_\rho^{\log}}\frac{1}{\log\left(e+\frac{\rho(x_0)}{r}\right)}.
		\end{align*}
		This, together with \eqref{06:55, 12/10/2023} and supp$\,\a(f-f_B)\subset B$, implies that $\a(f-f_B)$ is $C\|f\|_{BMO_\rho^{\log}}$ times an $(H^1_\rho,s)$-atom of log-type related to $B$. Consequently, by Remark \ref{12:09, 22/10/2023}\eqref{20:22, 20/10/2023} and Lemma \ref{07:01, 14/10/2023},
		$$\|\a(f-f_B)\|_{H^1_\rho}\lesssim \|f\|_{BMO_\rho^{\log}}.$$
		Thus, 
		$$\|f\|_{BMO}+ \sup_{\text{$(H^1_\rho,q)$-atom $\a$ related to $B$}} \|\a(f-f_{B})\|_{H^1_\rho}\lesssim \|f\|_{BMO_\rho^{\log}}.$$
		
		Conversely, suppose that $\a(f-f_{B})\in H^1_\rho(\X)$ for all $(H^1_\rho,q)$-atoms $\a$ related to the balls $B$. It suffices to prove that
		\begin{equation}\label{12:10, 04/12/2023}
			\log\left(e+\frac{\rho(x_0)}{r}\right) MO(f, B_0)\lesssim \|f\|_{BMO}+ \sup_{\text{$(H^1_\rho,q)$-atom $\a$ related to $B$}} \|\a(f-f_{B})\|_{H^1_\rho}
		\end{equation}
		for all balls $B_0=B(x_0,r)$. Indeed, the case of $r\geq \rho(x_0)$ is trivial since $MO(f, B_0)\leq \|f\|_{BMO}$. For the case of $0<r<\rho(x_0)$, let $g_{x_0}:\X\to \R$ be as in \eqref{17:18, 11/10/2023} and let $\a_0=\frac{1}{2\mu(B_0)}(h-h_{B_0})\chi_{B_0}$ with $h=\sign (f-f_{B_0})$, then it is easy to see that $\a_0$ is an $(H^1_\rho,q)$-atom related to the ball $B_0$. Moreover, from Lemma \ref{17:37, 11/10/2023} and Lemma \ref{17:53, 11/10/2023}\eqref{14:00, 12/10/2023} and Theorem \ref{06:44, 14/10/2023}\eqref{14:53, 13/12/2023}, we obtain that
		\begin{align*}
			\log\left(e+\frac{\rho(x_0)}{r}\right) MO(f, B_0)&=2 \log\left(e+\frac{\rho(x_0)}{r}\right)\int_{\X} \a_0(x) (f(x)-f_{B_0}) d\mu(x)\\
			&\leq 4 (g_{x_0})_{B_0} \int_{\X} \a_0(x) (f(x)-f_{B_0}) d\mu(x)\\
			&\lesssim \int_{\X} |\a_0(x)| |f(x)-f_{B_0}| |g_{x_0}(x)- (g_{x_0})_{B_0}| d\mu(x) +\\
			&\hskip2cm +\left|\int_{\X} \a_0(x) (f(x)-f_{B_0}) g_{x_0}(x) d\mu(x)\right| \\
			&\lesssim \|f\|_{BMO}\|g_{x_0}\|_{BMO} + \|\a_0(f-f_{B_0})\|_{H^1_\rho} \|g_{x_0}\|_{BMO_\rho}\\
			&\lesssim \|f\|_{BMO} + \sup_{\text{$(H^1_\rho,q)$-atom $\a$ related to $B$}} \|\a(f-f_{B})\|_{H^1_\rho}.
		\end{align*}
		This proves \eqref{12:10, 04/12/2023}, and thus completes the proof of Lemma \ref{07:01, 12/10/2023}.
		
	\end{proof}

	\section{Proof of the main theorems}

	\begin{lemma}\label{14:58, 06/10/2023}
		Let $q\in (1,\infty)$ and $T$ be an $(\omega_1,\omega_2)_\rho$-Calder\'on-Zygmund operator. Then, there exists a positive constant $C$ such that
		\begin{enumerate}[\rm (i)]				
			\item\label{08:35, 18/10/2023} for any $(H^1_\rho,q)$-atom $\a$ related to the ball $B=B(x_0,r)$ and any $j\in\mathbb Z^+_0$,
			$$\|T\a\|_{L^q(S_j(B))}\leq C \left[\omega_1\left(\left((2\kappa)^{\frac{1}{k_0+1}}\right)^{-j}\right)+\omega_2((2\kappa)^{-j})\right] \mu((2\kappa)^{j}B)^{\frac{1}{q}-1};$$

			\item\label{09:59, 15/10/2023} for any $(H^1_\rho,q)$-atom $\a$ related to the ball $B=B(x_0,r)$ and any $g\in BMO(\X)$,			
			$$\|(g-g_B)T\a\|_{L^1}\leq C \|g\|_{BMO}$$
			whenever $\omega_1$ and $\omega_2$ satisfy the log-Dini condition. 
		\end{enumerate}
	\end{lemma}
	
	\begin{proof}		
		\eqref{08:35, 18/10/2023} The case $j=0$ is trivial since $T$ is bounded on $L^q(\X)$. When $j\in\mathbb Z^+$, for any $x\in S_j(B)$, we consider the following two cases:
		
		{\sl Case 1:} $0<r<\frac{\rho(x_0)}{4}$. Then $\int_{\X} \a(y)d\mu(y)=0$, and thus the condition \eqref{Calderon-Zygmund 2} gives
		\begin{align}\label{21:52, 29/10/2023}
			|T\a(x)|&=\left|\int_B (K(x,y)-K(x,x_0))\a(y)d\mu(y)\right| \nonumber\\
			&\lesssim  \int_B \frac{1}{\mu(B(x_0,d(x,x_0)))} \omega_2\left(\frac{d(y,x_0)}{d(x,x_0)}\right) |\a(y)|d\mu(y)\nonumber\\
			&\lesssim \omega_2((2\kappa)^{-j}) \mu((2\kappa)^j B)^{-1}.
		\end{align}
		
		{\sl Case 2:} $\frac{\rho(x_0)}{4}\leq r<\rho(x_0)$. Then, for any $y\in B$, it follows from \eqref{16:19, 05/10/2023} that
		$$d(x,y)\geq \frac{d(x,x_0)}{\kappa}-d(y,x_0)\geq \frac{d(x,x_0)}{\kappa} - \frac{d(x,x_0)}{(2\kappa)^j}\geq \frac{d(x,x_0)}{2\kappa}$$
		and
		$$B(x_0, d(x,x_0))\subset B(x,(2\kappa)^2 d(x,y)).$$
		This, together with the condition \eqref{Calderon-Zygmund 1}, Remark \ref{17:15, 24/10/2023}\eqref{17:16, 24/10/2023} and $\rho(x_0)\leq 4 r$, implies that
		\begin{align*}
			|K(x,y)|&\lesssim \frac{1}{\mu(B(x,d(x,y)))}\omega_1\left(\left(1+\frac{d(x,y)}{\rho(x)}\right)^{-1}\right)\\
			&\lesssim \frac{1}{\mu(B(x_0,d(x,x_0)))}\omega_1\left(\left(1+\frac{d(x,x_0)}{\rho(x_0)}\right)^{-\frac{1}{k_0+1}}\right)\\
			&\lesssim \omega_1\left(\left((2\kappa)^{\frac{1}{k_0+1}}\right)^{-j}\right) \mu((2\kappa)^j B)^{-1}.
		\end{align*}
		Therefore,
		\begin{align}\label{21:53, 29/10/2023}
			|T\a(x)|=\left|\int_B K(x,y) \a(y)d\mu(y)\right|\lesssim \omega_1\left(\left((2\kappa)^{\frac{1}{k_0+1}}\right)^{-j}\right) \mu((2\kappa)^j B)^{-1}.
		\end{align}
	
		Combining \eqref{21:52, 29/10/2023} and \eqref{21:53, 29/10/2023}, we obtain that
		\begin{equation}\label{09:22, 15/10/2023}
			\|T\a\|_{L^q(S_j(B))}\lesssim \left[\omega_1\left(\left((2\kappa)^{\frac{1}{k_0+1}}\right)^{-j}\right)+\omega_2((2\kappa)^{-j})\right] \mu((2\kappa)^{j}B)^{\frac{1}{q}-1}.
		\end{equation}
		
		\eqref{09:59, 15/10/2023} By the H\"older inequality, Lemma \ref{14:51, 06/10/2023}, \eqref{09:22, 15/10/2023} and Remark \ref{12:09, 22/10/2023}\eqref{16:28, 21/10/2023},		
		\begin{align*}
			\|(g-g_B)T\a\|_{L^1}&=  \sum_{j=0}^\infty \|(g-g_B)T\a\|_{L^1(S_j(B))}\\
			&\leq \sum_{j=0}^\infty \|g-g_B\|_{L^{q'}((2\kappa)^{j+1} B)}\|T\a\|_{L^q(S_j(B))}\\
			&\lesssim \sum_{j=0}^\infty (j+2)\mu((2\kappa)^{j+1} B)^{\frac{1}{q'}}\|g\|_{BMO}\times\\
			&\hskip1cm\times \left[\omega_1\left(\left((2\kappa)^{\frac{1}{k_0+1}}\right)^{-j}\right)+\omega_2((2\kappa)^{-j})\right] \mu((2\kappa)^{j}B)^{\frac{1}{q}-1}\\
			&\lesssim \|g\|_{BMO}\sum_{j=0}^\infty (j+1) \left[\omega_1\left(\left((2\kappa)^{\frac{1}{k_0+1}}\right)^{-j}\right)+\omega_2((2\kappa)^{-j})\right]  \\
			&\lesssim \|g\|_{BMO} (\|\omega_1\|_{\text{log-Dini}}+\|\omega_2\|_{\text{log-Dini}})\lesssim \|g\|_{BMO}.
		\end{align*}
		
	\end{proof}

	Now we are ready to give the proofs of the main theorems.

	\begin{proof}[Proof of Theorem \ref{14:43, 02/11/2023}]
		\eqref{17:52, 10/10/2023} By \cite[Proposition 3.2]{YZ}, it suffices to prove that
		$$\|T\a\|_{L^1}\lesssim 1$$
		for all $(H^1_\rho,2)$-atoms $\a$ related to the balls $B\subset \X$. Indeed, since the H\"older inequality and Lemma \ref{14:58, 06/10/2023}\eqref{08:35, 18/10/2023} and Remark \ref{12:09, 22/10/2023}\eqref{16:28, 21/10/2023}, we get
		\begin{align*}
			\|T\a\|_{L^1} &=\sum_{j=0}^\infty \int_{S_j(B)} |T\a(x)|d\mu(x)\\
			&\leq \sum_{j=0}^\infty \mu(S_j(B))^{\frac{1}{2}}\|T\a\|_{L^2(S_j(B))}\\
			&\lesssim  \sum_{j=0}^\infty \mu((2\kappa)^{j} B)^{\frac{1}{2}} \left[\omega_1\left(\left((2\kappa)^{\frac{1}{k_0+1}}\right)^{-j}\right)+\omega_2((2\kappa)^{-j})\right] \mu((2\kappa)^{j}B)^{\frac{1}{2}-1}\\
			&\lesssim (\|\omega_1\|_{\text{Dini}}+\|\omega_2\|_{\text{Dini}})\lesssim 1.
		\end{align*}

		\eqref{17:50, 10/10/2023} By \cite[Proposition 3.2]{YZ} and \cite[Theorem (4.1)]{CW}, it suffices to prove that
		$$\left|\int_{\X} g(x) T\a(x) d\mu(x)\right|\lesssim \|g\|_{BMO}$$
		for all $(H^1_\rho,2)$-atoms $\a$ related to the balls $B=B(x_0,r)\subset \X$ and all $g\in C_c(\X)$. Indeed, since $\int_{\X} T\a(x)d\mu(x)=\langle T^*1,\a\rangle =0$ and Lemma \ref{14:58, 06/10/2023}\eqref{09:59, 15/10/2023}, we get
		$$\left|\int_{\X} g(x) T\a(x) d\mu(x)\right| = \left|\int_{\X} (g(x)-g_B) T\a(x) d\mu(x)\right|\lesssim \|g\|_{BMO}.$$
	\end{proof}

	\begin{proof}[Proof of Theorem \ref{18:02, 10/10/2023}]
		\eqref{17:59, 10/10/2023} $\Rightarrow$ \eqref{18:01, 10/10/2023} Suppose that $T$ is bounded on $H^1_\rho(\X)$. For any $(H^1_\rho,q)$-atom $\a$ related to the ball $B=B(x_0,r)$, let $g_{x_0}:\X\to [0,\infty)$ be as in \eqref{17:18, 11/10/2023}, then it follows from Lemma \ref{17:37, 11/10/2023} and Lemma \ref{14:58, 06/10/2023}\eqref{09:59, 15/10/2023} that	
		\begin{align*}
			\log\left(e+\frac{\rho(x_0)}{r}\right) \left|\int_{\X}T\a(x)d\mu(x)\right|&\lesssim \left|(g_{x_0})_B\right| \left|\int_{\X}T\a(x)d\mu(x)\right|\\
			&\lesssim \left|\int_{\X}T\a(x) g_{x_0}(x)d\mu(x)\right| + \\
			&\hskip1.2cm +\left|\int_{\X}T\a(x) (g_{x_0}(x)-(g_{x_0})_B)d\mu(x)\right| \\
			&\lesssim \|T\|_{H^1_\rho\to H^1_\rho}\|\a\|_{H^1_\rho}\|g_{x_0}\|_{BMO_\rho} +  \|g_{x_0}\|_{BMO}\lesssim 1.
		\end{align*}
		
		\eqref{18:01, 10/10/2023} $\Rightarrow$ \eqref{18:00, 10/10/2023}  Suppose that
		\begin{equation}\label{14:36, 12/10/2023}
			\left|\int_{\X}T\a(x)d\mu(x)\right|\lesssim \frac{1}{\log\left(e+\frac{\rho(x_0)}{r}\right)}
		\end{equation}
		for all $(H^1_\rho,q)$-atoms $\a$ related to the balls $B=B(x_0,r)$. By Lemma \ref{07:01, 12/10/2023} and Corollary \ref{12:49, 04/12/2023}\eqref{12:50, 04/12/2023}, it suffices to prove that
		\begin{equation}\label{14:52, 12/10/2023}
			\|\a(T^*1-(T^*1)_B)\|_{H^1_\rho}\lesssim 1
		\end{equation}
		for all $(H^1_\rho,q)$-atoms $\a$ related to the balls $B=B(x_0,r)$. Indeed, let $s:=\min\{2,\frac{q+1}{2}\}\in (1,q)$ then it follows from Lemma \ref{17:53, 11/10/2023}\eqref{14:01, 12/10/2023} and Corollary \ref{12:49, 04/12/2023}\eqref{12:50, 04/12/2023} that
		\begin{equation}\label{14:48, 12/10/2023}
			\|\a(T^*1-(T^*1)_B)\|_{L^{s}}\lesssim  \mu(B)^{\frac{1}{s}-1}\|T^*1\|_{BMO}\lesssim \mu(B)^{\frac{1}{s}-1}.
		\end{equation}
		Moreover, by \eqref{14:36, 12/10/2023}, Remark \ref{12:09, 22/10/2023}\eqref{14:29, 12/10/2023} and Corollary \ref{12:49, 04/12/2023}\eqref{12:50, 04/12/2023}, we get
		\begin{align*}
			\left|\int_{\X} \a(x)(T^*1(x)-(T^*1)_B)d\mu(x)\right|&\leq \left|\int_{\X} T\a(x)d\mu(x)\right|+\left|(T^*1)_B\int_{\X}\a(x)d\mu(x)\right|\\
			&\lesssim \frac{1}{\log\left(e+\frac{\rho(x_0)}{r}\right)}+ \frac{\|T^*1\|_{BMO_\rho}}{\log\left(e+\frac{\rho(x_0)}{r}\right)}\\
			&\lesssim \frac{1}{\log\left(e+\frac{\rho(x_0)}{r}\right)}.
		\end{align*}
		This, together with \eqref{14:48, 12/10/2023} and supp$\, \a(T^*1-(T^*1)_B)\subset B$, implies that $\a(T^*1-(T^*1)_B)$ is $C$ times an $(H^1_\rho,s)$-atom of log-type related to $B$. Thus \eqref{14:52, 12/10/2023} holds due to Remark \ref{12:09, 22/10/2023}\eqref{20:22, 20/10/2023} and Lemma \ref{07:01, 14/10/2023}.

		\eqref{18:00, 10/10/2023}$\,\Rightarrow\,$\eqref{17:59, 10/10/2023} Suppose that $T^*1\in BMO_\rho^{\log}(\X)$. By \cite[Proposition 3.2]{YZ} and Lemma \ref{07:01, 14/10/2023}, we only need to establish that $T\a$ is $C$ times an $(H^1_\rho,2,\omega_1,\omega_2)$-molecule of log-type for all $(H^1_\rho,2)$-atoms $\a$ related to the balls $B=B(x_0,r)$. Indeed, by Lemma \ref{14:58, 06/10/2023}\eqref{08:35, 18/10/2023}, it suffices to prove that
		\begin{equation}\label{17:15, 08/12/2023}
			\left|\int_{\X} T\a(x) d\mu(x)\right|\lesssim \frac{1}{\log\left(e+\frac{\rho(x_0)}{r}\right)}.
		\end{equation}
	
		Let $g_{x_0}:\X\to [0,\infty)$ be as in \eqref{17:18, 11/10/2023}. As $T^*1\in BMO_\rho^{\log}(\X)$, Lemma \ref{07:01, 12/10/2023} gives
		$$\|\a(T^*1-(T^*1)_B)\|_{H^1_\rho}\lesssim 1.$$
		This, together with Lemma \ref{17:37, 11/10/2023} and Lemma \ref{17:53, 11/10/2023}\eqref{14:00, 12/10/2023}, implies that
		\begin{align*}
			&\log\left(e+\frac{\rho(x_0)}{r}\right) \left|\int_{\X} \a(x)(T^*1(x)-(T^*1)_B)d\mu(x)\right|\\
			&\hskip3cm \leq 2 |(g_{x_0})_B| \left|\int_{\X} \a(x)(T^*1(x)-(T^*1)_B)d\mu(x)\right|\\
			&\hskip3cm\lesssim \left|\int_{\X} \a(x)(T^*1(x)-(T^*1)_B) g_{x_0}(x)d\mu(x)\right| + \\
			&\hskip4cm+ \left|\int_{\X} \a(x)(T^*1(x)-(T^*1)_B) (g_{x_0}(x)-(g_{x_0})_B)d\mu(x)\right| \\
			&\hskip3cm\lesssim \|\a(T^*1-(T^*1)_B)\|_{H^1_\rho}\|g_{x_0}\|_{BMO_\rho} + \|T^*1\|_{BMO} \|g_{x_0}\|_{BMO}\lesssim 1.
		\end{align*}	
		Therefore, by Remark \ref{12:09, 22/10/2023}\eqref{14:29, 12/10/2023} and Corollary \ref{12:49, 04/12/2023}\eqref{12:50, 04/12/2023},
		\begin{align*}
			\left|\int_{\X} T\a(x) d\mu(x)\right|&=\left|\int_{\X} \a(x) T^*1(x) d\mu(x)\right|\\
			&\leq \left|\int_{\X} \a(x)(T^*1(x)-(T^*1)_B)d\mu(x)\right|+ \left|(T^*1)_B \int_{\X} \a(x)d\mu(x)\right|\\
			&\lesssim \frac{1}{\log\left(e+\frac{\rho(x_0)}{r}\right)}+ \frac{\|T^*1\|_{BMO_\rho}}{\log\left(e+\frac{\rho(x_0)}{r}\right)}\lesssim \frac{1}{\log\left(e+\frac{\rho(x_0)}{r}\right)}.
		\end{align*}
		This proves \eqref{17:15, 08/12/2023}, and thus ends the proof of Theorem \ref{18:02, 10/10/2023}.
				
	\end{proof}

\section{Appendix}

Let $\omega_1, \omega_2: [0,\infty)\to[0,\infty)$ be two modulus of continuity satisfying the Dini condition, $\rho$ be an admissible function on $\X$, and $s\in (1,\infty)$. Based on the ideas from \cite{BPQ, BHQ, BDK, BLL, DLPV22, GLP}, we define the generalized $(\omega_1,\omega_2,s)_\rho$-Calder\'on-Zygmund operators as follows.

\begin{definition}\label{20:36, 21/12/2023}
	A bounded linear operator $T: L^s(\X)\to L^s_{\text{weak}}(\X)$ is called a generalized $(\omega_1,\omega_2,s)_\rho$-Calder\'on-Zygmund operator if 
		$$Tf(x)=\int_{\X} K(x,y)f(y)d\mu(y)$$
	for all $f\in L^s(\X)$ with bounded	support and all $x\notin\,$supp$\,f$, where the kernel $K(x,y)$ satisfies the condition: there exists a positive constant $C$ such that 
	\begin{equation}\label{generalized Calderon-Zygmund 1}
		\left(\int_{R\leq d(x,x_0)< 2\kappa R} \left|K(x,y)\right|^s d\mu(x)\right)^{\frac{1}{s}}\leq C \frac{\omega_1\left(\left(1+\frac{R}{\rho(x_0)}\right)^{-\frac{1}{k_0+1}}\right)}{\mu(B(x_0,R))^{1-\frac{1}{s}}}
	\end{equation}	
	and
	\begin{equation}\label{generalized Calderon-Zygmund 2}
		\left(\int_{R\leq d(x,x_0)< 2\kappa R} |K(x,y)-K(x,x_0)|^s d\mu(x)\right)^{\frac{1}{s}} \leq C \frac{\omega_2\left(\frac{d(y,x_0)}{R}\right)}{\mu(B(x_0,R))^{1-\frac{1}{s}}}
	\end{equation}
	for all $y, x_0\in\X$ and all $R> 2\kappa d(y,x_0)$.
\end{definition}

\begin{remark}\label{21:02, 21/12/2023}
	\begin{enumerate}[\rm (i)]
		\item\label{21:03, 21/12/2023} The condition \eqref{generalized Calderon-Zygmund 2} implies the standard H\"ormander condition 
		\begin{align*}
			\int_{d(x,x_0)>2\kappa d(y,x_0)} \left|K(x,y)-K(x,x_0)\right| d\mu(x) \lesssim \sum_{j=1}^{\infty}\omega_2((2\kappa)^{-j})\lesssim \|\omega_2\|_{\text{Dini}}<\infty
		\end{align*}
		for all $y, x_0\in\X$. Therefore, $T$ is of weak type $(1,1)$, and thus is bounded on $L^q(\X)$ for every $q\in (1,s)$.
		
		\item\label{21:06, 21/12/2023} If $1<q<s<\infty$ and $T$ is a generalized $(\omega_1,\omega_2,s)_\rho$-Calder\'on-Zygmund operator, then $T$ is a generalized $(\omega_1,\omega_2,q)_\rho$-Calder\'on-Zygmund operator. Furthermore, if $T$ is an $(\omega_1,\omega_2)_\rho$-Calder\'on-Zygmund operator as in Definition \ref{10:32, 02/11/2023}, then $T$ is a generalized $(\omega_1,\omega_2,s)_\rho$-Calder\'on-Zygmund operator for every $s\in (1,\infty)$.
		
		\item When $\X=\R^n$ and $\rho=1$, if $T$ is a generalized inhomogeneous $(\varepsilon,\delta,s)$-Calder\'on-Zygmund operator as in \cite{DLPV22}, then $T$ is a generalized $(\omega_1,\omega_2,s)_\rho$-Calder\'on-Zygmund of log-Dini type associated to $\omega_1(t)= (\alpha_1+\log\frac{1}{t})^{-\alpha_1}\chi_{(0,1]}(t)+\alpha_1^{-\alpha_1}\chi_{(1,\infty)}(t)$ and $\omega_2(t)= (\alpha_2+\log\frac{1}{t})^{-\alpha_2}\chi_{(0,1]}(t)+\alpha_2^{-\alpha_2}\chi_{(1,\infty)}(t)$ for any $\alpha_1,\alpha_2>2$. However, the reverse is not true in general.

		\item When $\X=\R^n$ and $\rho=\rho_L$, if $T$ is a generalized $(\delta,s)$-Schr\"odinger-Calder\'on-Zygmund operator as in \cite{BPQ, BHQ, BDK, BLL}, then $T$ is a generalized $(\omega_1,\omega_2,s)_\rho$-Calder\'on-Zygmund of log-Dini type associated to $\omega_1(t)= (\alpha_1+\log\frac{1}{t})^{-\alpha_1}\chi_{(0,1]}(t)+\alpha_1^{-\alpha_1}\chi_{(1,\infty)}(t)$ and $\omega_2(t)= (\alpha_2+\log\frac{1}{t})^{-\alpha_2}\chi_{(0,1]}(t)+\alpha_2^{-\alpha_2}\chi_{(1,\infty)}(t)$ for any $\alpha_1,\alpha_2>2$. However, the reverse is not true in general. 
	\end{enumerate}
	
\end{remark}

\begin{lemma}\label{20:50, 21/12/2023}
	Let $s\in (1,\infty)$, $q\in (1,s)$ and  $T$ be a generalized $(\omega_1,\omega_2,s)_\rho$-Calder\'on-Zygmund operator. Then, there exists a positive constant $C$ such that
	\begin{enumerate}[\rm (i)]
		\item\label{20:51, 21/12/2023} for any $(H^1_\rho,q)$-atom $\a$ related to the ball $B=B(x_0,r)$ and any $j\in\mathbb Z^+_0$,
		$$\|T\a\|_{L^q(S_j(B))}\leq C \left[\omega_1\left(\left((2\kappa)^{\frac{1}{k_0+1}}\right)^{-j}\right)+\omega_2((2\kappa)^{-j})\right] \mu((2\kappa)^{j}B)^{\frac{1}{q}-1};$$
		
		\item\label{20:54, 21/12/2023} for any $(H^1_\rho,q)$-atom $\a$ related to the ball $B=B(x_0,r)$ and any $g\in BMO(\X)$,			
		$$\|(g-g_B)T\a\|_{L^1}\leq C \|g\|_{BMO}$$
		whenever $\omega_1$ and $\omega_2$ satisfy the log-Dini condition. 
	\end{enumerate}
\end{lemma}

\begin{proof}
	
	\eqref{08:35, 18/10/2023} The case $j=0$ is trivial since $T$ is bounded on $L^q(\X)$ (see Remark \ref{21:02, 21/12/2023}\eqref{21:03, 21/12/2023}). When $j\in\mathbb Z^+$, we consider the following two cases:
	
	{\sl Case 1:} $0<r<\frac{\rho(x_0)}{4}$. Then $\int_{\X} \a(y)d\mu(y)=0$. Thus, from the H\"older inequality, $\|\a\|_{L^1}\leq 1$, the Fubini theorem, Remark \ref{21:02, 21/12/2023}\eqref{21:06, 21/12/2023} and the condition \eqref{generalized Calderon-Zygmund 2}, we get
	\begin{align}\label{22:41, 29/10/2023}
		\|T\a\|_{L^q(S_j(B))}&=  \left[\int_{(2\kappa)^j r\leq d(x,x_0)<(2\kappa)^{j+1} r} \left|\int_B (K(x,y)-K(x,x_0))\a(y)d\mu(y)\right|^q d\mu(x)\right]^{\frac{1}{q}}\nonumber\\
		&\leq \left[\int_B |\a(y)|d\mu(y) \int_{(2\kappa)^j r\leq d(x,x_0)<(2\kappa)^{j+1} r} |K(x,y)-K(x,x_0)|^q d\mu(x) \right]^{\frac{1}{q}}\nonumber\\
		&\lesssim \omega_2((2\kappa)^{-j}) \mu((2\kappa)^j B)^{\frac{1}{q}-1}.
	\end{align}
	
	{\sl Case 2:} $\frac{\rho(x_0)}{4}\leq r< \rho(x_0)$. Then, from the H\"older inequality, $\|\a\|_{L^1}\leq 1$, the Fubini theorem, Remark \ref{21:02, 21/12/2023}\eqref{21:06, 21/12/2023}, the condition \eqref{generalized Calderon-Zygmund 1} and $\rho(x_0)\sim r$, we get
	\begin{align}\label{22:42, 29/10/2023}
		\|T\a\|_{L^q(S_j(B))}&\leq \left[\int_{(2\kappa)^j r\leq d(x,x_0)<(2\kappa)^{j+1} r} \left(\int_B |K(x,y)| |\a(y)|d\mu(y)\right)^q d\mu(x)\right]^{\frac{1}{q}}\nonumber\\
		&\leq \left[\int_B |\a(y)|d\mu(y) \int_{(2\kappa)^j r\leq d(x,x_0)<(2\kappa)^{j+1} r} |K(x,y)|^q d\mu(x) \right]^{\frac{1}{q}}\nonumber\\
		&\lesssim \frac{\omega_1\left(\left(1+\frac{(2\kappa)^j r}{\rho(x_0)}\right)^{-\frac{1}{k_0+1}}\right)}{\mu((2\kappa)^j B)^{1-\frac{1}{q}}} \lesssim \omega_1\left(\left((2\kappa)^{\frac{1}{k_0+1}}\right)^{-j}\right) \mu((2\kappa)^j B)^{\frac{1}{q}-1}.
	\end{align}
	
	Combining \eqref{22:41, 29/10/2023} and \eqref{22:42, 29/10/2023}, we obtain that
	\begin{equation}\label{22:45, 29/10/2023}
		\|T\a\|_{L^q(S_j(B))}\lesssim \left[\omega_1\left(\left((2\kappa)^{\frac{1}{k_0+1}}\right)^{-j}\right)+\omega_2((2\kappa)^{-j})\right] \mu((2\kappa)^{j}B)^{\frac{1}{q}-1}.
	\end{equation}
	
	\eqref{09:59, 15/10/2023} From the H\"older inequality, Lemma \ref{14:51, 06/10/2023}, \eqref{22:45, 29/10/2023}  and Remark \ref{12:09, 22/10/2023}\eqref{16:28, 21/10/2023}, we get	
	\begin{align*}
		\|(g-g_B)T\a\|_{L^1}&=  \sum_{j=0}^\infty \|(g-g_B)T\a\|_{L^1(S_j(B))}\\
		&\leq \sum_{j=0}^\infty \|g-g_B\|_{L^{q'}((2\kappa)^{j+1} B)}\|T\a\|_{L^q(S_j(B))}\\
		&\lesssim \sum_{j=0}^\infty (j+2)\mu((2\kappa)^{j+1} B)^{\frac{1}{q'}}\|g\|_{BMO}\times\\
		&\hskip2cm\times \left[\omega_1\left(\left((2\kappa)^{\frac{1}{k_0+1}}\right)^{-j}\right)+\omega_2((2\kappa)^{-j})\right] \mu((2\kappa)^{j}B)^{\frac{1}{q}-1}\\
		&\lesssim \|g\|_{BMO}\sum_{j=0}^\infty (j+1) \left[\omega_1\left(\left((2\kappa)^{\frac{1}{k_0+1}}\right)^{-j}\right)+\omega_2((2\kappa)^{-j})\right]  \\
		&\lesssim \|g\|_{BMO} (\|\omega_1\|_{\text{log-Dini}}+\|\omega_2\|_{\text{log-Dini}})\lesssim \|g\|_{BMO}.
	\end{align*}
	
\end{proof}

\begin{theorem}\label{17:12, 24/10/2023}
	Let $s\in (1,\infty)$ and $T$ be a generalized $(\omega_1,\omega_2, s)_\rho$-Calder\'on-Zygmund operator. Then:
	\begin{enumerate}[\rm (i)]
		\item\label{17:52, 10/10/2023} $T$ is bounded from $H^1_\rho(\X)$ into $L^1(\X)$.
		
		\item\label{17:50, 10/10/2023} $T$ is bounded from $H^1_\rho(\X)$ into $H^1(\X)$ if and only if $T^*1=0$ whenever $\omega_1$ and $\omega_2$ satisfy the log-Dini condition. 
	\end{enumerate}
\end{theorem}

\begin{proof}
	By using Lemma \ref{20:50, 21/12/2023} instead of Lemma \ref{14:58, 06/10/2023}, the proof of Theorem \ref{17:12, 24/10/2023} is similar to the one of Theorem \ref{14:43, 02/11/2023}. We omit the details of the proof here.
\end{proof}

\begin{corollary}
	Under the same hypothesis as in Theorem \ref{17:12, 24/10/2023}. Then:
	\begin{enumerate}[\rm (i)]
		\item $T^*$ is bounded from $L^\infty(\X)$ into $BMO(\X)$. 
		
		\item $T^*$ is bounded from $BMO(\X)$ into $BMO_\rho(\X)$ if and only if $T^*1=0$ whenever $\omega_1$ and $\omega_2$ satisfy the log-Dini condition. 
	\end{enumerate}
\end{corollary}

\begin{theorem}\label{17:55, 26/10/2023}
	Let $s\in (1,\infty)$, $q \in (1,s)$ and $T$ be a generalized $(\omega_1,\omega_2,s)_\rho$-Calder\'on-Zygmund operator with $\omega_1, \omega_2$ satisfying the log-Dini condition. Then, the following three statements are equivalent:
	\begin{enumerate}[\rm (i)]
		\item $T$ is bounded on $H^1_\rho(\X)$;
		
		\item There exists a positive constant $C$ such that
		$$\left|\langle T^*1,\a\rangle\right|=\left|\int_{\X}T\a(x)d\mu(x)\right|\leq  \frac{C}{\log\left(e+\frac{\rho(x_0)}{r}\right)}$$
		for all $(H^1_\rho,q)$-atoms $\a$ related to the balls $B=B(x_0,r)$;
		
		\item $T^*1\in BMO_\rho^{\log}(\X)$, that is, there exists a positive constant $C$ such that		
		$$\frac{1}{\mu(B)}\int_{B} \left|T^*1(x)-(T^*1)_B\right|d\mu(x)\leq  \frac{C}{\log\left(e+\frac{\rho(x_0)}{r}\right)}$$		
		for all balls $B=B(x_0,r)$.		
	\end{enumerate}
\end{theorem}

\begin{proof}
	By using Lemma \ref{20:50, 21/12/2023} instead of Lemma \ref{14:58, 06/10/2023}, the proof of Theorem \ref{17:55, 26/10/2023} is similar to the one of Theorem \ref{18:02, 10/10/2023}. We omit the details of the proof here.
\end{proof}

\begin{corollary}
	Under the same hypothesis as in Theorem \ref{17:55, 26/10/2023}. Then, the following three statements are equivalent:
	\begin{enumerate}[\rm (i)]
		\item $T^*$ is bounded on $BMO_\rho(\X)$;
		
		\item There exists a positive constant $C$ such that
		$$\left|\langle T^*1,\a\rangle\right|=\left|\int_{\X} T\a(x)d\mu(x)\right|\leq  \frac{C}{\log\left(e+\frac{\rho(x_0)}{r}\right)}$$
		for all $(H^1_\rho,q)$-atoms $\a$ related to the balls $B=B(x_0,r)$;
		
		\item $T^*1\in BMO_\rho^{\log}(\X)$, that is, there exists a positive constant $C$ such that		
		$$\frac{1}{\mu(B)}\int_{B} \left|T^*1(x)-(T^*1)_B\right|d\mu(x)\leq  \frac{C}{\log\left(e+\frac{\rho(x_0)}{r}\right)}$$		
		for all balls $B=B(x_0,r)$.		
	\end{enumerate}
\end{corollary}

Finally, we give the proof of Lemma \ref{17:37, 11/10/2023} by proving the following proposition.

\begin{proposition}\label{14:47, 18/10/2023}
	There exists a positive constant $C$ such that
	\begin{equation}\label{15:00, 18/10/2023}
		\|\log d(\cdot,x_0)\|_{BMO}\leq C
	\end{equation}
	and
	\begin{equation}\label{15:01, 18/10/2023}
		\|g_{x_0}\|_{BMO_\rho}\leq C
	\end{equation}
	for all $x_0\in \X$, where the function $g_{x_0}$ is as in \eqref{17:18, 11/10/2023}.
\end{proposition}

\begin{proof}
	Observe that \eqref{15:00, 18/10/2023} is equivalent to the following
	$$\sup_{B}\frac{1}{\mu(B)}\inf_{C_B} \int_B \left|\log d(y,x_0)-C_B\right|d\mu(y)\lesssim 1,$$
	where the supremum is taken over all balls $B\subset\X$. Thus, it suffices to show that 
	\begin{equation}\label{16:01, 17/10/2023}
		\frac{1}{\mu(B)} \int_B \left|\log d(y,x_0)-C_B\right|d\mu(y)\lesssim 1
	\end{equation}
	for all balls $B=B(x,r)\subset \X$ with
	$$C_B=\begin{cases}
		\log d(x,x_0),& d(x,x_0)\geq 2\kappa r,\\
		\log r,& d(x,x_0) < 2\kappa r.
	\end{cases}$$
	Indeed, let us consider the following two cases:
	
	{\sl Case 1:} $d(x,x_0)\geq 2\kappa r$. Then, for any $y\in B=B(x,r)$, we have
	$$d(y,x_0)\leq \kappa(d(y,x)+d(x,x_0))<\frac{2\kappa+1}{2} d(x,x_0)$$
	and
	$$d(y,x_0)\geq \frac{1}{\kappa} d(x,x_0)-d(y,x)>\frac{1}{\kappa} d(x,x_0)-r\geq \frac{1}{2\kappa} d(x,x_0).$$
	Therefore,
	\begin{align*}
		\frac{1}{\mu(B)} \int_B \left|\log d(y,x_0)-C_B\right|d\mu(y)&=\frac{1}{\mu(B)} \int_B \left|\log d(y,x_0)- \log d(x,x_0)\right|d\mu(y)\\
		&\leq \log(2\kappa).
	\end{align*}
	
	{\sl Case 2:} $d(x,x_0)< 2\kappa r$. Then,
	$$B(x,r)\subset B(x_0, \kappa(2\kappa +1)r)\subset B(x, \kappa(2\kappa+\kappa(2\kappa+1))r).$$
	Therefore
	$$\mu(B(x,r))\leq \mu(B(x_0, \kappa(2\kappa +1)r))\leq \mu(B(x, \kappa(2\kappa+\kappa(2\kappa+1))r))\lesssim \mu(B(x,r))$$
	and, let $n_0:=\lfloor \log_2(\kappa(2\kappa+1))\rfloor +1$,
	\begin{align*}
		\int_B \left|\log d(y,x_0)-C_B\right|d\mu(y)&\leq \int_{B(x_0, 2^{n_0} r)} \left|\log d(y,x_0)- \log r\right|d\mu(y)\\
		&=\sum_{j=0}^{\infty} \int_{B(x_0, 2^{n_0-j}r)\setminus B(x_0, 2^{n_0-j-1}r)}\left|\log\frac{d(y,x_0)}{r}\right|d\mu(y)\\
		&\leq \sum_{j=0}^{\infty} (n_0+1+j)\mu(B(x_0, 2^{n_0-j}r)).
	\end{align*}
	This, together with \eqref{16:53, 14/10/2023}, implies that
	\begin{align*}
		\frac{1}{\mu(B)} &\int_B \left|\log d(y,x_0)-C_B\right|d\mu(y)\\
		&\lesssim \frac{1}{\mu(B(x_0, 2^{n_0}r))} \sum_{j=0}^{\infty} (n_0+1+j) C_0 2^{-j{\mathfrak d}}\mu(B(x_0, 2^{n_0}r))\lesssim 1,
	\end{align*}
	which proves \eqref{16:01, 17/10/2023}. Thus \eqref{15:00, 18/10/2023} holds.
	
	Noting that $\left\|\max\left\{0,1+\log\frac{\rho(x_0)}{d(\cdot,x_0)}\right\}\right\|_{BMO}\lesssim \|\log d(\cdot,x_0)\|_{BMO}\lesssim 1$, so \eqref{15:01, 18/10/2023} will be reduced to showing that  
	\begin{equation}\label{18:29, 16/10/2023}
		\frac{1}{\mu(B(x,r))}\int_{B(x,r)} g_{x_0}(y) d\mu(y)\lesssim 1
	\end{equation}
	for all balls $B(x,r)\subset \X$ with $r\geq \rho(x)$. First we show that
	\begin{equation}\label{18:14, 16/10/2023}
		\frac{1}{\mu(B(x_0, \rho(x_0)))}\int_{B(x_0, 3\rho(x_0))} g_{x_0}(y)d\mu(y)\lesssim 1.
	\end{equation}
	Indeed, by \eqref{16:53, 14/10/2023},
	\begin{align*}
		\int_{B(x_0,\rho(x_0))} \log\frac{\rho(x_0)}{d(y,x_0)}d\mu(y)&=\sum_{j=0}^{\infty} \int_{B(x_0, 2^{-j}\rho(x_0))\setminus B(x_0, 2^{-j-1}\rho(x_0))} \log\frac{\rho(x_0)}{d(y,x_0)}d\mu(y)\\
		&\leq \sum_{j=0}^{\infty} \mu(B(x_0, 2^{-j}\rho(x_0))) \log 2^{j+1}\\
		&\leq \sum_{j=0}^{\infty} C_0 2^{-j{\mathfrak d}}\mu(B(x_0,\rho(x_0)))(j+1)\\
		&\lesssim \mu(B(x_0,\rho(x_0))).
	\end{align*}
	Consequently, 
	\begin{align*}
		\int_{B(x_0, 3\rho(x_0))} g_{x_0}(y)d\mu(y) &= \int_{B(x_0, 3\rho(x_0))} \max\left\{0,1+\log\frac{\rho(x_0)}{d(y,x_0)}\right\}d\mu(y)\\
		&\leq \int_{B(x_0,\rho(x_0))}\left(1+\log\frac{\rho(x_0)}{d(y,x_0)}\right)d\mu(y) + \int_{B(x_0, 3\rho(x_0))\setminus B(x_0,\rho(x_0))}d\mu(y)\\
		&\lesssim \mu(B(x_0,\rho(x_0))),
	\end{align*}
	which proves \eqref{18:14, 16/10/2023}.
	
	We come now to prove \eqref{18:29, 16/10/2023}.	As supp$\,g_{x_0}\subset B(x_0,3\rho(x_0))$, we only need to consider the case of $B(x,r)\cap B(x_0,3\rho(x_0))\ne\emptyset$. Then, we have
	\begin{equation}\label{18:58, 16/10/2023}
		B(x_0, 3\rho(x_0))\subset B(x, (\kappa+\kappa^2)3\rho(x_0)+\kappa^2 r)
	\end{equation}
	and
	$$d(x,x_0)\leq \kappa(r+3\rho(x_0)).$$
	Therefore, by \eqref{admissible function},
	\begin{align*}
		\rho(x_0)\leq c_0 \rho(x)\left(1+\frac{d(x,x_0)}{\rho(x_0)}\right)^{k_0}\leq c_0 r \left(1+\frac{\kappa(r+3\rho(x_0))}{\rho(x_0)}\right)^{k_0},
	\end{align*}
	and thus
	$$\frac{r}{\rho(x_0)}\left(1+3\kappa+\kappa\frac{r}{\rho(x_0)}\right)^{k_0}\geq \frac{1}{c_0}.$$
	Consequently,
	$$\frac{r}{\rho(x_0)}\geq \min\left\{1,\frac{1}{c_0(1+4\kappa)^{k_0}}\right\}= \frac{1}{c_0(1+4\kappa)^{k_0}}.$$
	This, together with \eqref{18:58, 16/10/2023}, allows us to conclude that
	$$\mu(B(x_0,\rho(x_0)))\leq \mu(B(x_0,3\rho(x_0)))\leq \mu(B(x,C r))\lesssim \mu(B(x,r)).$$
	Hence, by \eqref{18:14, 16/10/2023}, $g_{x_0}\geq 0$ and supp$\,g_{x_0}\subset B(x_0,3\rho(x_0))$, we get
	$$\frac{1}{\mu(B(x,r))}\int_{B(x,r)} g_{x_0}(y) d\mu(y)\lesssim \frac{1}{\mu(B(x_0,\rho(x_0)))} \int_{B(x_0,3\rho(x_0))} g_{x_0}(y) d\mu(y)\lesssim 1.$$
	This proves \eqref{18:29, 16/10/2023}, and thus ends the proof of Proposition \ref{14:47, 18/10/2023}.
	
\end{proof}

	\begin{acknowledgement}
		The author would like to thank the referees for their careful reading and helpful suggestions. The paper was completed when the author visited Vietnam Institute for Advanced Study in Mathematics (VIASM). He would like to thank the VIASM for financial support and hospitality. This research is funded by Vietnam National Foundation for Science and Technology Development (NAFOSTED) under grant number 101.02-2023.09.
	\end{acknowledgement}
	
	\textbf{Conflict of interest:} The author states that there is no conflict of interest.

\end{document}